\documentclass{article}
\usepackage{amssymb,pstricks}
\usepackage{latexsym}
\usepackage{amsmath}
\usepackage{amsthm}

\topskip  \theoremstyle{remark}
\theoremstyle{plain}
\newtheorem{theorem}{Theorem}[section]
\newtheorem{lemma}[theorem]{Lemma}

\newtheorem{proposition}[theorem]{Proposition}

\newtheorem{corollary}[theorem]{Corollary}

\newtheorem{conjecture}[theorem]{Conjecture}

\begin{document}
\title{Distinct flattened partitions avoiding a pattern of length four}
\author{Toufik Mansour\\
\small Department of Mathematics, University of Haifa, 3498838 Haifa, Israel\\[-0.8ex]
\small\texttt{tmansour@univ.haifa.ac.il}\\[1.8ex]
Olivia Nabawanda\\
\small Department of Mathematics, Makerere University, Kampala, Uganda \\[-0.8ex]
\small\texttt{nabawandaolivia100@gmail.com}\\[1.8ex]
Mark Shattuck\\
\small Department of Mathematics, University of Tennessee, 37996 Knoxville, TN\\[-0.8ex]
\small\texttt{mark.shattuck2@gmail.com}\\[1.8ex]
}
\date{\small }
\maketitle

\begin{abstract}
Let $\mathcal{P}_n$ denote the set of distinct permutations of length $n$ that arise from the flattening process applied to the partitions of $[n]=\{1,\ldots,n\}$.  In this paper, we consider the problem of avoidance of a single classical pattern of length four by members of $\mathcal{P}_n$. Let $p_n(\tau)$ denote the number of members of $\mathcal{P}_n$ that avoid the pattern $\tau$.  We show that $p_n(\tau)=C_{n-1}$ for all $n \geq 1$ for seven patterns of length four yielding new combinatorial interpretations of the Catalan number sequence.  Further, we show that $p_n(\tau)$ corresponds to the binomial transform of Catalan numbers for three other patterns.  To establish our results, we suitably refine the counting sequence $p_n(\tau)$ in each case so as to obtain a system of functional equations satisfied by the corresponding generating functions.  These functional equations may then be solved explicitly leading to a determination of $p_n(\tau)$ in each case.

\end{abstract}



\section{Introduction}

Let $\pi=\pi_1\cdots \pi_n$ and $\tau=\tau_1\cdots\tau_k$ be permutations of length $n$ and $k$, where $1 \leq k \leq n$, expressed using the one-line notation.  Then $\pi$ is said to \emph{contain} $\tau$ if there is a subsequence of $\pi$ that is order-isomorphic to $\tau$, i.e., there exist indices $1 \leq i_1<\cdots<i_k \leq n$ such that $\pi_{i_a}>\pi_{i_b}$ if and only if $\tau_a>\tau_b$ for all $1 \leq a,b \leq k$.  If no such subsequence exists, then $\pi$ is said to \emph{avoid} $\tau$, and in this context, $\tau$ is often referred to as a (classical) \emph{pattern}.  By a \emph{partition} of a finite set, it is meant a collection of nonempty, pairwise disjoint subsets, called \emph{blocks}, whose union is the set. Let $\mathcal{R}_n$ denote the set of all partitions of $[n]$.  In this paper, we consider the pattern avoidance problem on a certain kind of sequential representation for  members of $\mathcal{R}_n$.

Let $\Pi=B_1/B_2/\cdots/B_k \in \mathcal{R}_n$, where the elements within each block $B_i$ are in increasing order and blocks are arranged from left to right in ascending order of their smallest elements.  Such a partition is said to be in \emph{standard form}.  Callan \cite{Callan} defined the \emph{flattening} of $\Pi$, denoted by $\text{flatten}(\Pi)$, as the permutation of $[n]$ obtained by arranging the blocks of $\Pi$ in standard form, removing the block separators and considering the resulting word.  For example, if $n=9$ and $\Pi=\{6,4,3\},\,\{5\},\,\{1,2,9\},\,\{8,7\}$ is a member of $\mathcal{R}_9$, then the standard form of $\Pi$ is $\{1,2,9\},\,\{3,4,6\},\,\{5\},\,\{7,8\}$ and $\text{flatten}(\Pi)=129346578$. Note that the flattening operation is in general not injective, as distinct set partitions can produce the same permutation.  In this work, we treat the resulting permutation of length $n$ as a combinatorial object in its own right and refer to it as a \emph{distinct flattened partition}.

Let $\mathcal{P}_n$ denote the set of all distinct flattened partitions of $[n]$.  For example, if $n=4$, then we have $\mathcal{P}_4=\{1234,1243,1324,1342,1423\}$.  In \cite{NRB}, a bijection was defined between $\mathcal{P}_{n+1}$ and $\mathcal{R}_{n}$, and hence $|\mathcal{P}_{n+1}|=B_{n}$ for all $n \geq 0$, where $B_n$ denotes the $n$-th Bell number.  A \emph{run} of elements within a permutation is a maximal sequence of consecutive increasing entries.  For example, if $\sigma=178263495$, then there are four runs, namely, 178, 26, 349 and 5.  Note that $\mathcal{P}_n$ is seen to comprise the subset of permutations of length $n$ whose elements that start runs form an increasing subsequence.  This structure is inherited from the standard form of the underlying set partitions.

The first systematic study of pattern avoidance in flattened structures was undertaken by Callan \cite{Callan}, who initiated the enumeration of set partitions according to the avoidance of patterns in their flattenings.  Specifically, for a pattern $\tau$, he considered the subset of $\mathcal{R}_n$
whose members when flattened avoid $\tau$ and determined the counting sequences for all patterns $\tau$ of length three. Liu and Zhang \cite{LZ} extended the results of Callan to several patterns of length four, partially addressing a problem posed by Mansour in \cite[Research Direction 3.2]{Man}.  An analogous notion of flattening has also been considered on permutations wherein the role of the blocks is played by cycles expressed in standard form.  See respectively \cite{MS1} and \cite{MS3} for comparable results concerning the avoidance of a pattern of length three by flattened permutations  and derangements.

Instead of considering restrictions on the underlying set of partitions,  Nabawanda and Rakotondrajao \cite{NR} looked at the avoidance problem on the set $\mathcal{P}_n$ of distinct permutations that arise from the flattening process.  Given a pattern $\tau$, let $\mathcal{P}_n(\tau)$ denote the subset of $\mathcal{P}_n$ whose members avoid $\tau$ and let $p_n(\tau)=|\mathcal{P}_n(\tau)|$ for $n \geq 1$.  In \cite{NR}, the authors considered the avoidance of one or more patterns of length three by members of $\mathcal{P}_n$.  We summarize in Table \ref{tab:length3-patterns} below the formulas that were found in \cite{NR} for $p_n(\tau)$, where $\tau \in \mathcal{S}_3$.  Here, $C_n=\frac{1}{n+1}\binom{2n}{n}$ and $M_n$ denote the $n$-th Catalan and Motzkin numbers, respectively, and all formulas hold for $n \geq 1$, except in the case of $213/312$, where $n \geq 2$. In this paper, we extend these results to $\tau \in \mathcal{S}_4$ and determine the complete Wilf-classification for the avoidance of a single pattern of length four by distinct flattened partitions up to one conjectured equivalence (see Table \ref{tab:we-classes} in the final section).

\begin{table}[h!]
	\centering
	\renewcommand{\arraystretch}{1.3}
	\begin{tabular}{|c|c|c|}
		\hline
		\textbf{Pattern $\tau$} & \textbf{Counting sequence $p_n(\tau)$} & \textbf{OEIS sequence} \\
		\hline
		$123$ & $1, 1, 1, 0, 0, 0, \ldots$ & Trivial \\
		\hline
		$132$ & $1, 1, 1, \ldots$ & Trivial \\
		\hline
		$213, 312$ & $2^{n-2}$ & \texttt{A011782}  \\
		\hline
		 $231$ & $M_{n-1}$ & \texttt{A001006} \\
		\hline
		$321$ & $C_{n-1}$ & \texttt{A000108}  \\
		\hline
	\end{tabular}
	\caption{Enumeration of $\tau$-avoiding flattened partitions for patterns of length 3}
	\label{tab:length3-patterns}
\end{table}

Members of $\mathcal{P}_n$ have also recently been studied with regard to their containment/avoidance of consecutive patterns, and generating function formulas enumerating members of $\mathcal{P}_n$ according to the number of occurrences of a pattern of length three were found in \cite{MS2}.  See also \cite{NRB}, where the distribution of the statistic tracking the number of runs on $\mathcal{P}_n$ was studied.  Variants of the notion of flattening have also been recently considered on other discrete structures, such as Catalan words \cite{BHR}, Stirling permutations \cite{BEF} and parking functions \cite{EHM}.  Note that the flattening idea seems to have originated with the work of Carlitz \cite{Car}, who considered flattened permutations in his definition of a certain kind of inversion statistic on $\mathcal{S}_n$.

The main results of this paper are given in Theorems \ref{Catalan} and \ref{trCatalan} below and are proven over the next three sections.

\begin{theorem}\label{Catalan}
We have $p_n(\tau)=C_{n-1}$ for all $n \geq 1$ and each $\tau \in \{1432,2143,2314,2413,3142,3412,4132\}$.
\end{theorem}

\begin{theorem}\label{trCatalan}
We have $p_n(\tau)=\sum_{k=0}^{n-2}\binom{n-2}{k}C_k$ for all $n \geq2$ and each $\tau \in \{3214,4213,4312\}$.
\end{theorem}

Note that Theorem \ref{Catalan}  follows from combining Theorems \ref{3142t1} and \ref{2143th}, Corollaries \ref{2413cor1} and \ref{2314cor1} and Proposition \ref{genprop} below (together with the fact $p_n(321)=C_{n-1}$).  From Theorem \ref{Catalan}, one obtains some apparently new combinatorial interpretations for the Catalan number sequence in terms of pattern avoidance.  See \cite{Stan} for a comprehensive list of discrete structures already known to be enumerated by $C_n$.  Combining Corollary \ref{3214cor} with Theorem \ref{4213th} below yields Theorem \ref{trCatalan}, which provides some new structures enumerated by the binomial transform of Catalan sequence (see A007317 in \cite{Sloane}).  We will make use of the generating function formulas
$$\sum_{n\geq 0}C_{n}x^n=\frac{1-\sqrt{1-4x}}{2x} \quad \text{and} \quad \sum_{n\geq0}\sum_{k=0}^{n}\binom{n}{k}C_kx^n=\frac{1-x-\sqrt{1-6x+5x^2}}{2x(1-x)},$$
and often denote the former by $C=C(x)$.

The organization of this paper is as follows.  In the next two sections, we show $p_n(\tau)=C_{n-1}$ for $\tau \in \{2413,3412,3142,2143,2314\}$, and in the fourth section, we prove Theorem \ref{trCatalan}.  Note that these results are obtained as specific cases of more general formulas involving the distributions of certain statistics on the respective avoidance classes.   In the final section, some enumerative results for singleton Wilf-classes are found and a general criterion is given for when the patterns $\tau$ and $1(\tau+1)$ are equivalent on $\mathcal{P}_n$.  Furthermore, a couple of conjectures are raised, one concerning an equivalence and another an enumeration.  Our results for the enumeration/ classification of patterns of length four are summarized in Table  \ref{tab:we-classes}.

We briefly describe now our methods as follows.  To establish our results, we first consider various refinements of the sequences $p_n(\tau)$ based on certain statistics on $\mathcal{P}_n$ which are specific to the pattern at hand.  These refinements lead to arrays with two or three parameters, for which one is able to write a recurrence (whereas finding a recurrence a priori for the sequence $p_n(\tau)$ by direct combinatorial reasoning may not be possible).  The statistics that proved useful here in identifying an appropriate array for writing a recurrence included ones tracking the second or last letter, the last two letters or the final letter in the first run.  Still other parameters were needed in finding a recurrence once a suitable array had been identified, such as the length of the first run or the letter directly following $n$ (if it exists).  We then defined certain generating functions, often based on the form of the recurrence, for the aforementioned arrays  and/or their row sums.  Translating the recurrences to generating functions, sometimes a difficult step, then leads to a single functional equation or a system.  We then proceed to find an explicit solution of the functional equation(s) in its full generality, using the kernel method (see, e.g., \cite{HM}) or other techniques.  Taking all of the auxiliary variables to be unity in this solution then yields a formula for the generating function of the sequence $p_n(\tau)$, from which one may deduce the desired avoidance result.

\section{The patterns 2413, 3412 and 3142}

We first establish the cases 2413 and 3412, which may be treated simultaneously due to the similarity of their arguments.  Recall that a \emph{descent} within a sequence $\sigma=\sigma_1\cdots \sigma_m$ corresponds to an index $i \in [m-1]$ such that $\sigma_i>\sigma_{i+1}$.  The actual entries $\sigma_i$ and $\sigma_{i+1}$ are referred to as the \emph{top} and \emph{bottom} of the descent.  Note that $\mathcal{P}_n$ consists of those permutations of length $n$ that start with 1 whose descent bottoms from left to right form an increasing sequence.  In particular, no member of $\mathcal{P}_n$ contains two consecutive descents.

Let $u_n=p_n(2413)$ and $v_n=p_n(3412)$ for $n \geq 1$.  In order to write recurrences, we consider the arrays $u_{n,i}$ and $v_{n,i}$ enumerating respectively the members of $\mathcal{U}_n=\mathcal{P}_n(2413)$ and $\mathcal{V}_n=\mathcal{P}_n(3412)$  having last letter $i$ for $1 \leq i \leq n$. Note that $u_{n,1}=v_{n,1}=\delta_{n,1}$ for $n \geq 1$.

The $u_{n,i}$ and $v_{n,i}$ are given recursively as follows.

\begin{lemma}\label{2413lem1}
If $n \geq 4$, then
\begin{equation}\label{2413lem1e1}
u_{n,i}=u_{n-1,i}+\sum_{j=2}^{i-1}\sum_{m=0}^{n-i-1}u_{j+m,j}u_{n-m-j,i-j+1}, \qquad 3 \leq i \leq n-1,
\end{equation}
and
\begin{equation}\label{2413lem1e2}
v_{n,i}=v_{n-1,i}+\sum_{j=2}^{i-1}\sum_{m=0}^{i-j-1}v_{i-m-1,j}v_{n+m-i+1,m+2}, \qquad 3 \leq i \leq n-1,
\end{equation}
with $u_{n,2}=v_{n,2}=1$, $u_{n,n}=u_{n-1}$ and $v_{n,n}=v_{n-1}$ for all $n \geq 2$.
\end{lemma}
\begin{proof}
Let $\mathcal{U}_{n,i}$ and $\mathcal{V}_{n,i}$ denote respectively the subsets of $\mathcal{U}_n$ and $\mathcal{V}_n$ enumerated by $u_{n,i}$ and $v_{n,i}$.  Note that $\mathcal{U}_{n,2}$ and $\mathcal{V}_{n,2}$ are singletons consisting of the sequence $134\cdots n2$, whereas members of $\mathcal{U}_{n,n}$ and $\mathcal{V}_{n,n}$ are synonymous with those in $\mathcal{U}_{n-1}$ and $\mathcal{V}_{n-1}$, which implies the conditions stated for $i=2,n$.  So assume $n \geq 4$, where $3 \leq i \leq n-1$, and we first show \eqref{2413lem1e1}.  To do so, consider whether or not $n$ is the penultimate letter of $\pi \in \mathcal{U}_{n,i}$ with $3 \leq i \leq n-1$.  If so, then $n$ may clearly be deleted yielding $u_{n-1,i}$ possibilities.  If not, then consider the letter $j$ directly succeeding $n$ within $\pi$.  Note $j$ being a descent bottom implies $j \in [2,i-1]$.  Suppose $\pi=\alpha nj\beta i$ (*), where $\beta$ is possibly empty.  Since $j$  is a descent bottom, all letters in $[j-1]$ must occur in $\alpha$.  Further, the set of letters in $[j+1,n-1]$ occurring within $\alpha$ must comprise an interval of the form $[n-m,n-1]$, for otherwise, a 2413 would occur wherein the roles of the `4' and `1' are played by $n$ and $j$.  Note that $i$ occurring to the right of $j$ implies $0 \leq m \leq n-i-1$.  Thus, the sections $\alpha j$ and $j\beta i$ are 2413-avoiding flattened partitions comprising the letters in $[j]\cup[n-m,n-1]$ and $[j,n-m-1]$, respectively.  Upon standardizing, we have $\alpha j \in \mathcal{U}_{j+m,j}$ and $j\beta i \in \mathcal{U}_{n-m-j,i-j+1}$.

Furthermore, we have that all $\pi$ of the form (*), where $\alpha$ and $\beta$ are as described above, indeed avoid 2413.  To see this, first note by requiring that $\alpha j$ (and not just $\alpha$ by itself) avoid 2413, we eliminate the possibility of a 2413 arising when the sections $\alpha$ and $\beta$ are combined to form $\pi$ as in (*). If such a 2413 were to arise, the `2' and `1' would be played by letters in $[j-1]$, with the `3' belonging to $[j,n-m-1]$.  Note that indeed any letter in $[j,n-m-1]$ could serve as the `3' due to its lying to the right of $\alpha$, and hence $\alpha j$ avoiding such a 2413 implies $\pi$ does as well.  Moreover, there can clearly be no occurrence of 2413 in $\pi$ of the form (*) with $\alpha,\,\beta$ as stated wherein an element of $[n-m,n-1]$ serves as the `2', as all letters to the right of $n$ are less than $n-m$.  Thus, the conditions stated in the preceding paragraph are both necessary and sufficient for $\pi$ of the form (*) to avoid 2413.  Considering all possible $j$ and $m$ then implies that the double summation expression on the right-hand side of \eqref{2413lem1e1} enumerates all members of $\mathcal{U}_{n,i}$ whose penultimate letter is not $n$.  When taken together with the first case above, one obtains recurrence \eqref{2413lem1e1}.

A similar argument applies to \eqref{2413lem1e2}.  Suppose $\pi \in \mathcal{V}_{n,i}$ whose penultimate letter is not $n$ is decomposed as in (*).  Then all letters in $\beta i$ must now be greater than those in $\alpha$, for otherwise there would be a 3412 wherein the roles of  `4' and `1' are played by $n$ and $j$.  Therefore, the letters in $\beta i$ must comprise $[n-m,n-1]$ for some $m \geq 1$.  Note $j+1 \leq n-m \leq i$ implies $n-i \leq m \leq n-j-1$.  Analogous to before, the requirement that the sections $\alpha j$ and $j\beta i$ both avoid 3412, with $\beta i$ comprising $[n-m,n-1]$ for some $m$ as given, is seen to be both necessary and sufficient for $\pi$ to avoid 3412.  Note that within $j\beta i$, the letter $n-m$ serves as the 2, and hence $i$ acts as $i-(n-m-2)$, in this flattened partition of size $m+1$.  Thus, the sections $\alpha j$ and $j \beta i$ on their respective alphabets are seen to correspond to members of $\mathcal{V}_{n-m-1,j}$ and $\mathcal{V}_{m+1,i-n+m+2}$. Considering all possible $j$ and $m$, and replacing $m$ by $m+n-i$, then accounts for the summation expression in \eqref{2413lem1e2} and completes the proof.
\end{proof}

Define the distribution polynomials $u_n(y)=\sum_{i=1}^nu_{n,i}y^{i-1}$ and $v_n(y)=\sum_{i=1}^nv_{n,i}y^{i-1}$ for $n \geq 1$.

\begin{lemma}\label{2413lem2}
If $n \geq 3$, then
\begin{equation}\label{2413lem2e1}
u_n(y)=y^{n-1}u_{n-1}(1)+u_{n-1}(y)+\sum_{m=2}^{n-2}u_m(y)u_{n-m}(y)
\end{equation}
and
\begin{equation}\label{2413lem2e2}
v_n(y)=y^{n-1}v_{n-1}(1)+v_{n-1}(y)+\sum_{m=2}^{n-2}y^{n-m-1}v_m(y)v_{n-m}(1),
\end{equation}
with $u_1(y)=v_1(y)=1$ and $u_2(y)=v_2(y)=y$.
\end{lemma}
\begin{proof}
Multiplying both sides of \eqref{2413lem1e1} by $y^{i-1}$, and summing over $3 \leq i \leq n-1$, yields
\begin{align*}
u_n(y)&=y+u_{n-1}y^{n-1}+\sum_{i=3}^{n-1}u_{n-1,i}y^{i-1}+\sum_{i=3}^{n-1}y^{i-1}\sum_{j=2}^{i-1}\sum_{m=0}^{n-i-1}u_{j+m,j}u_{n-m-j,i-j+1}\\
&=y^{n-1}u_{n-1}(1)+u_{n-1}(y)+\sum_{j=2}^{n-2}\sum_{i=j+1}^{n-1}y^{i-1}\sum_{m=0}^{n-i-1}u_{j+m,j}u_{n-m-j,i-j+1}.
\end{align*}
Working on the preceding sum, we have
\begin{align*}
&\sum_{j=2}^{n-2}\sum_{i=j+1}^{n-1}y^{i-1}\sum_{m=0}^{n-i-1}u_{j+m,j}u_{n-m-j,i-j+1}\\
&=\sum_{j=2}^{n-2}\sum_{m=0}^{n-j-2}u_{j+m,j}\sum_{i=j+1}^{n-m-1}u_{n-m-j,i-j+1}y^{i-1}\\
&=\sum_{j=2}^{n-2}\sum_{m=0}^{n-j-2}u_{j+m,j}y^{j-1}\sum_{i=2}^{n-m-j}u_{n-m-j,i}y^{i-1}=\sum_{j=2}^{n-2}\sum_{m=0}^{n-j-2}y^{j-1}u_{j+m,j} u_{n-m-j}(y)\\
&=\sum_{j=2}^{n-2}\sum_{m=j}^{n-2}y^{j-1}u_{m,j}u_{n-m}(y)=\sum_{m=2}^{n-2}u_{n-m}(y)\sum_{j=2}^mu_{m,j}y^{j-1}=\sum_{m=2}^{n-2}u_m(y)u_{n-m}(y),
\end{align*}
which implies \eqref{2413lem2e1}. Formula \eqref{2413lem2e2} may be shown in a comparable manner starting with \eqref{2413lem1e2}.
\end{proof}

Let $U(x,y)=\sum_{n\geq1}u_n(y)x^n$ and $V(x,y)=\sum_{n\geq1}v_n(y)x^n$.

\begin{theorem}\label{2413th1}
We have
\begin{equation}\label{2413th1e1}
U(x,y)=\frac{1+x-\sqrt{1-4x+x^2+2x\sqrt{1-4xy}}}{2}
\end{equation}
and
\begin{equation}\label{2413th1e2}
V(x,y)=\frac{x(3y-2+2x-2xy-y\sqrt{1-4xy})}{2(y-1+x)}.
\end{equation}
\end{theorem}
\begin{proof}
Multiplying both sides of \eqref{2413lem2e1} by $x^n$, and summing over all $n \geq3$, we obtain
\begin{align}
U(x,y)&=x+x^2y+x(U(xy,1)-xy)+x(U(x,y-x)+\sum_{n\geq4}x^n\sum_{m=2}^{n-2}u_m(y)u_{n-m}(y)\notag\\
&=x-x^2+x(U(xy,1)+U(x,y))+\sum_{m\geq2}u_m(y)\sum_{n\geq m+2}u_{n-m}(y)x^n\notag\\
&=x-x^2+x(U(xy,1)+U(x,y))+(U(x,y)-x)^2.\label{2413th1e3}
\end{align}
Taking $y=1$ in \eqref{2413th1e3} implies $p=U(x,1)$ satisfies $p^2-p+x$, and hence $p=\frac{1-\sqrt{1-4x}}{2}=xC$, where the choice of the negative root is dictated by the condition $U(0,1)=0$.  Substituting $U(xy,1)=\frac{1-\sqrt{1-4xy}}{2}$ into \eqref{2413th1e3} implies
$$U(x,y)^2-(1+x)U(x,y)+\frac{x}{2}\left(3-\sqrt{1-4xy}\right)=0,$$
and solving for $U(x,y)$ in the last equation gives \eqref{2413th1e1}.

Using \eqref{2413lem2e2}, one can show
\begin{equation}\label{2413th1e4}
V(x,y)=x-x^2+x(V(xy,1)+V(x,y))+\frac{1}{y}(V(x,y)-x)(V(xy,1)-xy).
\end{equation}
Taking $y=1$ in \eqref{2413th1e4} implies $V(x,1)=xC$, as before.  Substituting this back into \eqref{2413th1e4}, and solving for $V(x,y)$, leads to \eqref{2413th1e2} after some algebra.
\end{proof}

Letting $y=1$ in Theorem \ref{2413th1} yields the following result.

\begin{corollary}\label{2413cor1}
We have $U(x,1)=V(x,1)=\frac{1-\sqrt{1-4x}}{2}=xC$, and hence $p_n(2413)=p_n(3412)=C_{n-1}$ for all $n \geq 1$.
\end{corollary}

For the remainder of this section, we focus on the pattern 3142. Let $w_n=p_n(3142)$ for $n \geq 1$.   Given $n \geq 2$ and $2 \leq i \leq n$, let $w_{n,i}$ enumerate those members of $\mathcal{W}_n=\mathcal{P}_{n}(3142)$ whose second letter is $i$.  Note that $w_n=\sum_{i=2}^nw_{n,i}$ for $n \geq 2$, with $w_1=1$.

The $w_{n,i}$ are given recursively as follows.

\begin{lemma}\label{3142lem1}
If $n \geq 3$, then
\begin{equation}\label{3142lem1e1}
w_{n,i}=w_{i-2}w_{n-i+2}, \qquad 3 \leq i \leq n,
\end{equation}
with $w_{n,2}=w_{n-1}$ for $n\geq 2$.
\end{lemma}
\begin{proof}
Let $\mathcal{W}_{n,i}$ denote the subset of $\mathcal{W}_n$ whose members have second letter $i$ for $2 \leq i \leq n$.  Then members of  $\mathcal{W}_{n,2}$ are synonymous with those in $\mathcal{W}_{n-1}$, upon deletion of 2, which implies the initial condition for $i=2$.  Suppose now $\pi \in \mathcal{P}_n$ is expressible as
\begin{equation}\label{3142lem1e2}
\pi=1i\alpha2\beta\gamma,
\end{equation}
for some $i \in [3,n]$, where $\alpha$, $\beta$ or $\gamma$ may be empty, $\alpha$ is a subset of $[i+1,n]$ in increasing order, and the leftmost letter in $[i+1,n]$ occurring to the right of 2 within $\pi$ (if it exists) coincides with the first letter of $\gamma$.  Let $x$ denote the final letter of $2\beta$.  Then we have that $\pi$ of the form \eqref{3142lem1e2} avoids the pattern 3142 if and only if the subsequences $2\beta$ and $i\alpha x\gamma$ both avoid 3142, with $2\beta$ consisting of the letters in $[2,i-1]$. To see this, first note that no letter in $[3,i-1]$ can occur in $\gamma$,  whence these letters must all lie in $\beta$, for otherwise, there would be a 3142 as witnessed by $i2yz$, where $y$ denotes the first letter of $\gamma$ and $z$ is some letter in $\gamma$ that is less than $i$. Further, the letters in $\beta$ all being smaller than those in $\gamma$ rules out an occurrence of 3142 starting in $\beta$ and ending in $\gamma$.  Finally, by requiring that $i\alpha x \gamma$ avoid 3142, one rules out any occurrences of the pattern whose first letter lies to the left of 2 or to the right of $\beta$.

To enumerate members of $\mathcal{W}_{n,i}$ where $3 \leq i \leq n$, first note that there are clearly $w_{i-2}$ possibilities for $2\beta$.  Viewing $x$ as the second smallest letter within the 3142-avoiding flattened partition $i\alpha x\gamma$, one has that it represents an arbitrary member of $\mathcal{W}_{n-i+2}$.  Since the standardized forms of $2\beta$ and $i\alpha x\gamma$ in \eqref{3142lem1e2} may be chosen independently of one another, the equality \eqref{3142lem1e1} follows, which completes the proof.
\end{proof}

Let $w_n(y)=\sum_{i=2}^nw_{n,i}y^{i-2}$ for $n \geq 2$ and $W(x,y)=\sum_{n\geq 2}w_n(y)x^n$.

\begin{theorem}\label{3142t1}
We have
\begin{equation}\label{3142t1e1}
W(x,y)=\frac{1-(1+2x)\sqrt{1-4x}-(1-2x)\sqrt{1-4xy}+\sqrt{(1-4x)(1-4xy)}}{4},
\end{equation}
and, in particular, $\sum_{n\geq 2}p_n(3142)x^n=W(x,1)=x(C-1)$.
\end{theorem}
\begin{proof}
By \eqref{3142lem1e1}, we have
\begin{align}
w_n(y)&=w_{n-1}+\sum_{i=3}^ny^{i-2}w_{i-2}w_{n-i+2}\notag\\
&=(1+y)w_{n-1}(1)+\sum_{i=2}^{n-2}y^iw_i(1)w_{n-i}(1), \qquad n \geq 3, \label{3142t1e2}
\end{align}
with $w_2(y)=1$. Multiplying both sides of \eqref{3142t1e2} by $x^n$, and summing over all $n \geq 3$, yields
\begin{align}
W(x,y)-x^2&=(1+y)\sum_{n\geq 3}w_{n-1}(1)x^n+\sum_{n\geq 4}x^n\sum_{i=2}^{n-2}y^iw_i(1)w_{n-i}(1) \notag\\
&=x(1+y)W(x,1)+\sum_{i\geq2}y^iw_i(1)\sum_{n\geq i+2}w_{n-i}(1)x^n\notag\\
&=x(1+y)W(x,1)+W(x,1)W(xy,1). \label{3142t1e3}
\end{align}
Taking $y=1$ in \eqref{3142t1e3} implies $p=W(x,1)$ satisfies $p^2-(1-2x)p+x^2=0$, and hence $p=x(C-1)$.  Substituting this back into \eqref{3142t1e3}, and simplifying, yields \eqref{3142t1e1} and completes the proof.
\end{proof}

\section{2143 and 2314}

We first consider the case of avoiding 2143. Given $n \geq 2$ and $2 \leq i \leq n$, let $a_{n,i}$ denote the number of members of $\mathcal{P}_n(2143)$ whose second letter is $i$.  Let $a_n=p_n(2143)$, and hence $a_n=\sum_{i=2}^na_{n,i}$ for $n \geq 2$, with $a_1=1$.  In order to write recurrences, we consider the restriction $a_{n,i}^*$ of $a_{n,i}$ enumerating only those members of $\mathcal{P}_n(2143)$ that start $1i2$, where $n \geq 3$ and $i \in [3,n]$.  For example, if $n=5$ and $i=3$, then we have $a_{5,3}=4$ (the enumerated partitions being 13245, 13425, 13452, 13524) and $a_{5,3}^*=1$ (accounted for by 13245).

The arrays $a_{n,i}$ and $a_{n,i}^*$ satisfy the following system of recurrences.

\begin{lemma}\label{2143lem1}
If $n \geq 3$, then
\begin{equation}\label{2143lem1e1}
a_{n,i}=\sum_{j=0}^{n-i}\binom{n-i}{j}a_{n-j,i}^*, \qquad 3 \leq i \leq n,
\end{equation}
with $a_{n,2}=a_{n-1}$ for all $n \geq 2$.  If $n \geq 5$, then
\begin{equation}\label{2143lem1e2}
a_{n,i}^*=\sum_{m=0}^{n-i}a_{n-m-1,i-1}^*+\sum_{k=4}^{i-1}\sum_{j=0}^{i-k-1}\sum_{m=0}^{n-i}\binom{i-k-1}{j}a_{n-j-m-2,k-1}^*, \quad 4 \leq i \leq n-1,
\end{equation}
with $a_{n,3}^*=1$ and $a_{n,n}^*=a_{n-2}$ for $n \geq 3$.
\end{lemma}
\begin{proof}
Let $\mathcal{A}_{n,i}$ and $\mathcal{A}_{n,i}^*$ denote the subsets of $\mathcal{A}_n=\mathcal{P}_n(2143)$ whose members are enumerated by $a_{n,i}$ and $a_{n,i}^*$, respectively.  That $a_{n,2}=a_{n-1}$ is clear, so assume $\pi \in \mathcal{A}_{n,i}$ for some $3 \leq i \leq n$.  Considering the number $j$ of elements in $[i+1,n]$ occurring between $i$ and $2$ within $\pi$ yields \eqref{2143lem1e1}, as deletion of these $j$ elements is seen to result in an arbitrary member of $\mathcal{A}_{n-j,i}^*$ after standardization. Clearly, we have $a_{n,3}^*=1$ and $a_{n,n}^*=a_{n-2}$ for all $n \geq 3$, as $\mathcal{A}_{n,3}^*=\{13245\cdots n\}$ with members of $\mathcal{A}_{n,n}^*$ being synonymous with those in $\mathcal{A}_{n-2}$, upon deletion of the first two letters.

So assume $\pi \in \mathcal{A}_{n,i}^*$, where $4 \leq i \leq n-1$.  We consider cases based on the fourth letter $k$.  If $k=3$, then we may delete $k$ resulting in a member of $\mathcal{A}_{n-1,i-1}^*$.  If $k>i$, then we must have $k=i+1$, for otherwise $i2k(k-1)$ is an occurrence of 2143.  Furthermore, $\pi$ in this case must be of the form $\pi=1i2(i+1)\cdots(i+m)3\pi'$ for some $m \geq 1$, as the members of $[i+1,n]$ must occur in increasing order from left to right within $\pi$.  It is seen that each of the letters in $[i+1,i+m]$ is extraneous concerning the avoidance of 2143 within $\pi$, as is 3.  Deletion of these letters produces an arbitrary member of $\mathcal{A}_{n-m-1,i-1}^*$ for each $m$, and thus considering all possible $m$ yields $\sum_{m=1}^{n-i}a_{n-m-1,i-1}^*$ members of $\mathcal{A}_{n,i}^*$ for which $k=i+1$.  Combining with the prior case where $k=3$ then yields the first summation expression on the right side of \eqref{2143lem1e2}.

Finally, suppose $k \in [4,i-1]$.  In this case, we have $\pi=1i2k\alpha\beta3\pi'$, where $\alpha$ is a possibly empty sequence of letters in $[k+1,i-1]$ in increasing order and $\beta=(i+1)(i+2)\cdots(i+m)$ for some $m \geq0$.  We may delete the letters in $\alpha$ (due to their being preceded by $k$ and $2$) and also those in $\beta$ (due to their comprising an interval starting with $i+1$). Additionally, we have that the $i$ (and hence also 1) may be deleted, for if $i2ab$ is an occurrence of 2143 for some $a>b>i$, then so is $k3ab$.  This leaves the flattened partition $2k3\pi'$, which when standardized is seen to be an arbitrary member of $\mathcal{A}_{n-j-m-2,k-1}^*$, where $j=|\alpha|$.  Considering all possible $k$, $j$ and $m$ then accounts for the second summation in \eqref{2143lem1e2} and completes the proof.
\end{proof}

Let $A(x,v)=\sum_{n\geq2}\sum_{i=2}^na_{n,i}v^{i-2}x^n$
and $B(x,v)=\sum_{n\geq3}\sum_{i=3}^na_{n,i}^*v^{i-3}x^n$.

\begin{lemma}\label{2143lem2}
We have
\begin{align}
A(x,v)&=x^2+xA(x,1)+v(1-x)^2B\left(\frac{x}{1-x},v(1-x)\right)\label{2143lem2e1}
\end{align}
and
\begin{align}
&\left(1-\frac{vx}{1-x}-\frac{v^2x^2}{(1-x)(1-v-vx)}\right)B(x,v)\notag\\
&=\frac{x^3}{1-x}-vx^4+\frac{x^2}{v}(1-vx)A(vx,1)-x^2(1-vx)^2\left(\frac{1}{(1-x)(1-v-vx)}+\frac{1}{v}\right)B\left(\frac{vx}{1-vx},1-vx\right).\label{2143lem2e2}
\end{align}
\end{lemma}
\begin{proof}
By \eqref{2143lem1e1}, we have
$$A(x,v)-\sum_{n\geq 3}a_{n-1}x^n-x^2=\sum_{n\geq3}x^n\sum_{i=3}^nv^{i-2}\sum_{j=0}^{n-i}\binom{n-i}{j}a_{n-j,i}^*,$$
and hence
\begin{align*}
A(x,v)-xA(x,1)-x^2&=\sum_{i\geq3}v^{i-2}\sum_{n\geq i}x^n\sum_{j=i}^n\binom{n-i}{n-j}a_{j,i}^*=\sum_{i\geq3}\sum_{j\geq i}a_{j,i}^*v^{i-2}\sum_{n\geq j}\binom{n-i}{j-i}x^n\\
&=\sum_{i\geq3}\sum_{j\geq i}a_{j,i}^*v^{i-2}\cdot \frac{x^i}{(1-x)^{j-i+1}}=\sum_{j\geq3}\frac{x^j}{(1-x)^j}\sum_{i=3}^ja_{j,i}^*v^{i-2}(1-x)^{i-1}\\
&=v(1-x)^2B\left(\frac{x}{1-x},v(1-x)\right),
\end{align*}
which yields \eqref{2143lem2e1}.

By \eqref{2143lem1e2}, we have
\begin{align*}
&B(x,v)-\frac{x^3}{1-x}-\frac{x^2}{v}A(vx,1)=\sum_{n\geq5}\sum_{i=4}^{n-1}a_{n,i}^*v^{i-3}x^n\\
&=\sum_{n\geq5}\sum_{i=4}^{n-1}\Bigg(
  a_{n-1, i-1}^*
  + \sum_{m=1}^{n-i} a^*_{n-m-1, i-1}
  + \sum_{k=4}^{i-1}
      \sum_{j=0}^{i-k-1}
      \sum_{m=0}^{n-i}
        \binom{i-k-1}{j}
        a_{n-j-m-2, k-1}^*\Bigg)v^{i-3}x^{n}\\
&=vx\left(B(x,v)-\frac{x^2}{v}A(vx,1)-x^3\right)\\
&\quad+\sum_{n\geq5}
\sum_{i=4}^{n-1}
\Bigg(
    \sum_{m=1}^{\,n-i}
        a_{n-m-1, i-1}^* v^{i-3}x^{n}
    +
    \sum_{k=4}^{i-1}
        \sum_{j=0}^{i-k-1}
        \sum_{m=0}^{n-i}
            \binom{i-k-1}{j}
            a_{n-j-m-2, k-1}^*
            v^{i-3}x^{n}\Bigg).\\
\end{align*}
Working separately on the last two multi-sum expressions, we get
\begin{align*}
&\sum_{n\geq5}\sum_{i=4}^{n-1}\sum_{m=1}^{\,n-i}a_{n-m-1, i-1}^* v^{i-3}x^{n}=\sum_{i\geq4}\sum_{n\geq i+1}\sum_{m=i}^{n-1}a_{m-1,i-1}^*v^{i-3}x^n\\
&=\sum_{i\geq4}\sum_{m\geq i}a_{m-1,i-1}^*v^{i-3}\cdot\frac{x^{m+1}}{1-x}=\frac{x}{1-x}\sum_{m\geq4}\sum_{i=4}^ma_{m-1,i-1}^*v^{i-3}x^m=\frac{vx^2}{1-x}B(x,v)
\end{align*}
and
\begin{align*}
&\sum_{n\geq5}\sum_{i=4}^{n-1}\sum_{k=4}^{i-1}\sum_{j=0}^{i-k-1}\sum_{m=0}^{n-i}\binom{i-k-1}{j}a_{n-j-m-2, k-1}^*v^{i-3}x^{n}\\
&=\sum_{k\geq4}\sum_{i\geq1}\sum_{n\geq i+k+1}\sum_{j=0}^{i-1}\sum_{m=0}^{n-i-k}\binom{i-1}{j}a_{n-j-m-2,k-1}^*v^{i+k-3}x^n\\
&=\sum_{k\geq3}\sum_{i\geq1}\sum_{n\geq k+2}\sum_{j=0}^{i-1}\sum_{m=0}^{n-k-1}\binom{i-1}{j}a_{n+i-j-m-2,k}^*v^{i+k-2}x^{n+i}\\
&=\sum_{k\geq3}\sum_{n\geq k+2}\sum_{j\geq0}\sum_{i\geq1}\sum_{m=0}^{n-k-1}\binom{i+j-1}{j}a_{n+i-m-2,k}^*v^{i+j+k-2}x^{n+i+j}\\
&=\sum_{k\geq3}\sum_{i\geq1}\sum_{n\geq i+k+2}\sum_{m=0}^{n-i-k-1}a_{n-m-2,k}^*v^{i+k-2}x^n\sum_{j\geq0}\binom{i+j-1}{j}(vx)^j\\
&=\sum_{k\geq3}\sum_{i\geq1}\sum_{n\geq i+k+2}\sum_{m=i+k-1}^{n-2}\frac{a_{m,k}^*v^{i+k-2}x^n}{(1-vx)^i}.
\end{align*}

This last sum may be written as the difference of two sums as
\begin{align*}
&\sum_{k\geq3}\sum_{i\geq1}\sum_{n\geq i+k+1}\sum_{m=i+k-1}^{n-2}\frac{a_{m,k}^*v^{i+k-2}x^n}{(1-vx)^i}-\sum_{k\geq3}\sum_{i\geq1}\frac{a_{i+k-1,k}^*v^{i+k-2}x^{i+k+1}}{(1-vx)^i}\\
&=\sum_{k\geq3}\sum_{i\geq1}\sum_{m\geq i+k-1}\frac{a_{m,k}^*v^{i+k-2}x^{m+2}}{(1-x)(1-vx)^i}-\sum_{k\geq3}\sum_{i \geq k}\frac{a_{i,k}^*v^{i-1}x^{i+2}}{(1-vx)^{i-k+1}}\\
&=\sum_{k\geq3}\sum_{m\geq k}\frac{a_{m,k}^*v^{k-1}x^{m+2}}{1-x}\cdot\frac{1-\left(\frac{v}{1-vx}\right)^{m-k+1}}{1-v-vx}-\frac{x^2}{v}\sum_{i\geq3}\sum_{k=3}^ia_{i,k}^*\left(\frac{vx}{1-vx}\right)^i(1-vx)^{k-1}\\
&=\frac{v^2x^2}{(1-x)(1-v-vx)}\sum_{m\geq3}\sum_{k=3}^ma_{m,k}^*v^{k-3}x^m\left(1-\left(\frac{v}{1-vx}\right)^{m-k+1}\right)-\frac{x^2}{v}(1-vx)^2B\left(\frac{vx}{1-vx},1-vx\right)\\
&=\frac{v^2x^2}{(1-x)(1-v-vx)}\left(B(x,v)-\frac{(1-vx)^2}{v^2}B\left(\frac{vx}{1-vx},1-vx\right)\right)-\frac{x^2}{v}(1-vx)^2B\left(\frac{vx}{1-vx},1-vx\right).
\end{align*}
Combining with the prior computations, we get
\begin{align*}
&B(x,v)-\frac{x^3}{1-x}-\frac{x^2}{v}A(vx,1)\\
&=\frac{vx}{1-x}B(x,v)-x^3A(vx,1)-vx^4+\frac{v^2x^2}{(1-x)(1-v-vx)}B(x,v)\\
&\quad-x^2(1-vx)^2\left(\frac{1}{(1-x)(1-v-vx)}+\frac{1}{v}\right)B\left(\frac{vx}{1-vx},1-vx\right),
\end{align*}
which may be rewritten as \eqref{2143lem2e2}.
\end{proof}

It is possible to guess and then verify the solution to the system of functional equations \eqref{2143lem2e1} and \eqref{2143lem2e2}.  Note that reversing the steps of Lemma \ref{2143lem2} shows that this system has a unique power series solution.

\begin{theorem}\label{2143th}
We have
$$A(x,v)=\frac{x(1-\sqrt{1-4x})}{2}-\frac{vx^2(1-3x-(1-x)\sqrt{1-4vx})}{2(1-2x-v(1-x)^2)}$$
and
$$B(x,v)=-\frac{x^2(1-2x-\sqrt{1-4vx})}{2(1-v-x)}.$$
In particular, upon taking $v=1$, we have $p_n(2143)=C_{n-1}$ for all $n \geq 1$, with the cardinality of the subset of $\mathcal{P}_n(2143)$ whose members have third letter 2 given by $C_{n-2}$ for $n \geq 3$.
\end{theorem}

We now consider the case of avoiding 2314. Given $n \geq 2$ with $i,j \in [n]$ and $i \neq j$, let $b_{n,i,j}$ denote the number of members of $\mathcal{B}_n=\mathcal{P}_n(2314)$ whose last two letters are $i,j$. Note that $b_n=\sum_{i=1}^n\sum_{j=1,j\neq i}^nb_{n,i,j}$ for all $n \geq 2$, where $b_n=p_n(2314)$.   Clearly, we have $b_{n,i,1}=b_{n,1,j}=0$ if $n \geq 3$ for all $i$ and $j$, with $b_{2,1,2}=1$, $b_{2,2,1}=0$.

The $b_{n,i,j}$ satisfy the following recurrence relations.

\begin{lemma}\label{2314lem1}
If $n \geq 3$, then
\begin{equation}\label{2314l1e1}
b_{n,i,j}=\sum_{k=0}^{i-j-1}\sum_{s=1}^{j-1}\binom{i-j-1}{k}b_{n-k-1,s,i-k-1}, \quad 2 \leq j< i \leq n,
\end{equation}
\begin{equation}\label{2314l1e2}
b_{n,i,i+1}=\sum_{s=1,s\neq i}^{n-1}b_{n-1,s,i}, \qquad 2 \leq i \leq n-1,
\end{equation}
and
\begin{equation}\label{2314l1e3}
b_{n,i,j}=\sum_{m=0}^1\sum_{s=1}^{i-1}b_{n-m-1,s,j-m-1}+\sum_{m=0}^1\sum_{r=j+1}^n\sum_{k=0}^{r-j-1}\sum_{s=1}^{i-1}\binom{r-j-1}{k}b_{n-k-m-2,s,r-k-m-2},
\end{equation}
for $2 \leq i<j-1 \leq n-1$.
\end{lemma}
\begin{proof}
Let $\mathcal{B}_{n,i,j}$ denote the subset of $\mathcal{B}_n$ enumerated by $b_{n,i,j}$ and $\mathcal{B}_{n,j}=\cup_{i=1,i\neq j}^n\mathcal{B}_{n,i,j}$.  If $i>j$, then $\pi \in \mathcal{B}_{n,i,j}$ may be decomposed as $\pi=\pi's\alpha i j$, where $\alpha$ consists of letters in $[j+1,i-1]$ in ascending order and $s \in [j-1]$.  Then we may clearly delete the terminal $j$, due to $i>j$ directly preceding it, as well as the entirety of section $\alpha$, since these letters occur in ascending order and are of intermediate value between $s$ and $i$ (and hence none of them are actually needed in serving as a `1' or `4' in a possible occurrence of 2314). These deletions are seen to result after standardization in a member of $\mathcal{B}_{n-k-1,s,i-k-1}$, where $k=|\alpha|$.  As there are $\binom{i-j-1}{k}$ possibilities for $\alpha$, summing over all $k$ and $s$ yields \eqref{2314l1e1}.  Furthermore, within a member of $\mathcal{B}_{n,i,i+1}$, where $2 \leq i \leq n-1$, the final letter may clearly be deleted, resulting in an arbitrary member of $\mathcal{B}_{n-1,i}$ whose cardinality is given by the right side of \eqref{2314l1e2}.

To show \eqref{2314l1e3}, suppose $\pi \in \mathcal{B}_{n,i,j}$, where $n \geq 4$ and $2 \leq i <j-1 \leq n-1$.  Let $r$ denote the antepenultimate letter of $\pi$ and we consider cases on $r$.   If $r<i$, then we may delete $i$, yielding $\sum_{s=1}^{i-1}b_{n-1,s,j-1}$ possibilities. On the other hand, if $i<r<j$, then we must have $r=i+1$, for if $i+1<r<j$, then $(i+1)rij$ is a 2314.  Since the descent bottoms within a member of $\mathcal{P}_n$ are increasing, the predecessor of $i+1$ must be less than $i$.  In this case, we may delete both $i$ and $i+1$ as they are extraneous concerning the avoidance of 2314.  This leaves a member of $\mathcal{B}_{n-2,s,j-2}$ for some $s \in [i-1]$, and hence there are $\sum_{s=1}^{i-1}b_{n-2,s,j-2}$ possibilities for $\pi$.  The preceding two cases may be combined to give the first summation expression on the right-hand side of \eqref{2314l1e3}.

Now suppose $r \in [j+1,n]$.  Then we have $$\pi=\pi's\delta\beta rij,$$
where $s \in [i-1]$, $\delta$ is either empty or represents the letter $i+1$, and $\beta$ is a possibly empty sequence of letters in $[j+1,r-1]$ in increasing order.  Note that $\delta$ is as described due to $\pi$ ending in $i,j$, with $i<j-1$.  Further, the presence of both $s<i$ and $r>j$ within the increasing run directly preceding $i$ implies that $i$ and $j$ are redundant concerning the avoidance of 2314 (outside of the restriction placed on $\delta$).  Thus, we may delete both $i$ and $j$ and, in turn, all of the letters in $\delta$ and $\beta$ as well. Let $m=|\delta|$ and $k=|\beta|$, and hence $m=0$ or 1 and $0 \leq k \leq r-j-1$, with $\binom{r-j-1}{k}$ possibilities for $\beta$.  Then performing all of these deletions on $\pi$ is seen to result in an arbitrary member of $\mathcal{B}_{n-k-m-2,s,r-k-m-2}$.  Thus, for each $m$, $r$, $k$ and $s$, there are $\binom{r-j-1}{k}b_{n-k-m-2,s,r-k-m-2}$ possibilities for $\pi$.  Considering all possible values of these parameters accounts for the second summation in \eqref{2314l1e3} and completes the proof.
\end{proof}

It is convenient to consider the following two generating functions:
$$B_+(x,u,v)=\sum_{n\geq2}\sum_{i=1}^{n-1}\sum_{j=i+1}^nb_{n,i,j}u^{i-1}v^{j-i-1}x^n$$
and
$$B_-(x,u,v)=\sum_{n\geq2}\sum_{i=2}^{n}\sum_{j=1}^{i-1}b_{n,i,j}u^{i-1}v^{j-i-1}x^n.$$

Translation of the recurrences in Lemma \ref{2314lem1} (the details of which are lengthy and we omit) yields the following system of functional equations.

\begin{lemma}\label{2314lem2}
We have
\small\begin{align}
B_-(x,u,v)&=\frac{u^2x}{v(1-v+ux)}\left(\frac{1}{v-ux}B_+\left(x,u,\frac{u}{v-ux}\right)-B_+(x,u,u)\right),\label{2314lem2e1}\\
 B_+(x,u,v)&=x^2+\frac{uvx}{u-v}(B_+(x,u,u)-B_+(x,u,v))+\frac{uvx^2}{u-v}(uB_+(x,u,u)-vB_+(x,u,v))\notag\\ &\quad+\frac{ux^2}{(1-v+vx)(1-u+ux)(u-v)}\biggl(v(uvx-uv-u+v)\left(B_+(x,u,v)-B_+\left(x,u,\frac{1}{1-x}\right)\right)\notag\\
&\qquad\qquad\qquad+uv\left(B_+(x,u,v)-v(1-x)B_+\left(x,u,\frac{1}{1-x}\right)\right)-uv(1-v+vx)B_+(x,u,u)\biggr)\notag\\ &\quad+\frac{uvx^3}{(1-v+vx)(1-u+ux)(u-v)}\biggl((uvx-uv-u+v)\left(vB_+(x,u,v)-\frac{1}{1-x}B_+\left(x,u,\frac{1}{1-x}\right)\right)\notag\\
&\qquad\qquad\qquad+uv\left(B_+(x,u,v)-B_+\left(x,u,\frac{1}{1-x}\right)\right)-u^2(1-v+vx)B_+(x,u,u)\biggr)\notag\\
&\quad+ux(B_-(x,u,u)+B_+(x,u,u)). \label{2314lem2e2}\normalsize
\end{align}
\end{lemma}

By substituting \eqref{2314lem2e1} into \eqref{2314lem2e2}, we obtain after some algebra the following functional equation for $B_+(x,u,v)$:
\begin{align}
& (x - 1)(u x - u + 1)
  \bigl( u v^3 x^2 - u v^2 x^2 + u v^2 x - 2 u v x + v^2 x + u v - v^2 - u + v \bigr)B_+(x,u,v)\nonumber\\
& =  u^2 x(x -1 ) (v x - v + 1) (v x + 1) (u - 1) B_+(x,u,u)\nonumber\\
& \quad+ ux^2  (v x + 1) (u - v) \, B_+\left(x,u,\frac{1}{1-x}\right)
- x^2 (x - 1) (u x - u + 1) (v x - v + 1)(u - v).\label{eq2314a}
\end{align}

In order to solve \eqref{eq2314a}, we apply the kernel method and let $v_i=v_i(x,u)$ for $0 \leq i \leq 2$ denote the roots of the polynomial equation $u x^2v^3 - (1-x)(1-ux) v^2+(1+u- 2ux)v - u=0$.  The $v_i$ have the following series expansions (where we assume $0<|u|, |x|<1$):
\begin{align*}
v_0&=\frac{1}{u}x^{-2}-\frac{u+1}{u}x^{-1}-u-(u^2+1)x-(2u^3+u^2+u+1)x^2+\cdots,\\
v_1&=u+u^2x+\frac{u^3(2u-1)}{u-1}x^2+\frac{u^4(4u-1)}{u-1}x^3+\cdots,\\
v_2&=1+x-\frac{1}{u-1}x^2-\frac{2u+1}{u-1}x^3+\cdots.
\end{align*}
Substituting $v=v_1$ and $v=v_2$ into \eqref{eq2314a}, and solving the resulting system in the quantities $B_+(x,u,u)$ and $B_+\left(x,u,\frac{1}{1-x}\right)$, we obtain
\begin{align*}
B_+(x,u,u)&=\frac{(u-v_1)(u-v_2)x}{(1+v_1x)(1+v_2x)(1-u)u^2},\\
B_+\left(x,u,\frac{1}{1-x}\right)&=-\frac{(1-x)(1-v_1+v_1x)(1-v_2+v_2x)(1+ux)}{u(1+v_1x)(1+v_2x)}.
\end{align*}
Substituting these expressions into \eqref{eq2314a}, we obtain
\begin{equation}
B_+(x,u,v)=\frac{(v-v_1)(v-v_2)(1+ux)x^2}{(u+(2ux-u-1)v+(1-x)(1-ux)v^2-ux^2v^3)(1+v_1x)(1+v_2x)}.\label{2314b}
\end{equation}

Let $B(x,u,v)=\sum_{n\geq 2}\sum_{i=1}^n\sum_{j=1,j\neq i}^nb_{n,i,j}u^{i-1}v^{j-i-1}x^n$.  Note that $B(x,u,v)=B_+(x,u,v)+B_-(x,u,v)$, by the definitions.  From \eqref{2314lem2e1}, we obtain the following result.

\begin{theorem}\label{2314thm}
The generating function $B(x,u,v)$ equals
\begin{equation}\label{2314thme1}
B_+(x,u,v)+\frac{u^2x}{v(1-v+ux)}\left(\frac{1}{v-ux}B_+\left(x,u,\frac{u}{v-ux}\right)-B_+(x,u,u)\right),
\end{equation}
where $B_+(x,u,v)$ is given by \eqref{2314b}.
\end{theorem}

Using the facts $v_0v_1v_2=\frac{1}{x^2}$ and $v_0+v_1+v_2=\frac{(1-x)(1-ux)}{ux^2}$, one can obtain the following formula for $B(x,u,v)$ in terms of $v_0$ only.

\begin{theorem}\label{2314thm2}
The generating function $B(x,u,v)$ is  given by \eqref{2314thme1},
where
\begin{align*}
B_+(x,u,v)&=
\frac{
x (u x + 1)
\bigl(
u v^{2} v_{0} x^{2}
+ u v v_{0}^{2} x^{2}
- u v v_{0} x^{2}
+ u v v_{0} x
+ v v_{0} x
- v v_{0}
+ u
\bigr)
}{
\bigl(u v_{0}^{2} x^{2} - u v_{0} x^{2} - u x + v_{0} x - v_{0}\bigr)
\bigl(u v^{3} x^{2} - u v^{2} x^{2} + u v^{2} x - 2 u v x + v^{2} x + u v - v^{2} - u + v\bigr)
}.
\end{align*}
\end{theorem}

Note that one may take $u=v=1$ in Theorem \ref{2314thm2}.  Doing so, and noting $v_0\mid_{u=1}=\frac{1}{x^2}-\frac{2}{x}-C^2(x)$, we get
$$B(x,1,1)=\frac{1}{1-x}B_+\left(x,1,\frac{1}{1-x}\right)=\frac{1-2x-\sqrt{1-4x}}{2},$$
which yields the following result.

\begin{corollary}\label{2314cor1}
We have $B(x,1,1)=\frac{1-2x-\sqrt{1-4x}}{2}=x(C-1)$, and hence $p_n(2314)=C_{n-1}$ for all $n \geq 1$.
\end{corollary}

\noindent\emph{Remark:} One also obtains from Theorem \ref{2314thm2} the formula
$$B_+(x,1,1)=\frac{1-x-(1+x)\sqrt{1-4x}}{2(2+x)},$$
from which $B_-(x,1,1)$ may be found by subtraction from $B(x,1,1)$.

\section{3214, 4213 and 4312}

In this section, we prove Theorem \ref{trCatalan} and first treat the pattern 3214. Let $d_n=p_n(3214)$.  To determine $d_n$, we proceed in a manner analogous to that used in the derivation of the result for 2314.  Given $n \geq 2$ with $i,j \in [n]$ and $i \neq j$, let $d_{n,i,j}$ denote the number of members of $\mathcal{D}_n=\mathcal{P}_n(3214)$ whose last two letters are $i,j$. Note that $d_n=\sum_{i=1}^n\sum_{j=1,j\neq i}^nd_{n,i,j}$ for all $n \geq 2$.   Furthermore, $d_{n,i,1}=d_{n,1,j}=0$ if $n \geq 3$ for all $i$ and $j$, with $d_{2,1,2}=1$, $d_{2,2,1}=0$.

The $d_{n,i,j}$ are determined recursively as follows.

\begin{lemma}\label{3214lem1}
If $n \geq 3$, then
\begin{equation}\label{3214l1e1}
d_{n,i,j}=\sum_{k=0}^{i-j-1}\sum_{s=1}^{j-1}\binom{i-j-1}{k}d_{n-k-1,s,i-k-1}, \quad 2 \leq j< i \leq n,
\end{equation}
\begin{equation}\label{3214l1e2}
d_{n,i,i+1}=\sum_{s=1,s\neq i}^{n-1}d_{n-1,s,i}, \qquad 2 \leq i \leq n-1,
\end{equation}
and
\begin{equation}\label{3214l1e3}
d_{n,i,j}=\sum_{m=0}^{j-i-1}\sum_{s=1}^{i-1}d_{n-m-1,s,j-m-1}+\sum_{m=0}^{j-i-1}\sum_{r=j+1}^n\sum_{k=0}^{r-j-1}\sum_{s=1}^{i-1}\binom{r-j-1}{k}d_{n-k-m-2,s,r-k-m-2},
\end{equation}
for $2 \leq i<j-1 \leq n-1$.
\end{lemma}
\begin{proof}
Formulas \eqref{3214l1e1} and \eqref{3214l1e2} may be justified as \eqref{2314l1e1} and \eqref{2314l1e2} were above, so we focus on \eqref{3214l1e3}.  Let $\mathcal{D}_{n,i,j}$ denote the subset of $\mathcal{D}_n$ enumerated by $d_{n,i,j}$.  Suppose $\pi \in \mathcal{D}_{n,i,j}$, where $2 \leq i<j-1$. Let $s$ denote the rightmost letter of $\pi$ belonging to $[i-1]$ and let $r$ be the antepenultimate letter of $\pi$. If $r \in [i+1,j-1]$, then we must have $r=j-1$ in order to avoid 3214, and hence $\pi$ is expressible as $\pi=\pi's(j-m)(j-m+1)\cdots(j-1)ij$ for some $1 \leq m \leq j-i-1$.  Allowing $m$ to equal zero accounts for the case in which $r<i$.  Note that each of the letters $j-m,\ldots,j-1$ and $i$ may be deleted, leaving a member of $\mathcal{D}_{n-m-1,s,j-m-1}$.  Considering all possible $m$ and $s$ yields the first summation on the right side of \eqref{3214l1e3}.  If $r>j$, then the same argument applies as the one used in justifying the second summation in \eqref{2314l1e3} except that now the number of letters in $[i+1,j-1]$ occurring between $s$ and $i$ may range from 0 to $j-i-1$, with these letters being $j-m,j-m+1,\ldots,j-1$ for each $m$.  Thus, one need only replace $\sum_{m=0}^1$ with $\sum_{m=0}^{j-i-1}$ within the second summation in \eqref{2314l1e3} to account for those members of $\mathcal{D}_{n,i,j}$ wherein $i<j-1$ for which $r>j$.  Combining the prior cases then yields \eqref{3214l1e3} and completes the proof.
\end{proof}

Define the generating functions
$$D_+(x,u,v)=\sum_{n\geq2}\sum_{i=1}^{n-1}\sum_{j=i+1}^nd_{n,i,j}u^{i-1}v^{j-i-1}x^n$$
and
$$D_-(x,u,v)=\sum_{n\geq2}\sum_{i=2}^{n}\sum_{j=1}^{i-1}d_{n,i,j}u^{i-1}v^{j-i-1}x^n.$$

Translating the recurrences \eqref{3214l1e1}--\eqref{3214l1e3} to generating functions (we omit the details) yields the following functional equations.

\begin{lemma}\label{3214l2}
We have
\begin{align}
D_-(x,u,v)&=\frac{u^2x}{v(1-v+ux)}\left(\frac{1}{v-ux}D_+\left(x,u,\frac{u}{v-ux}\right)-D_+(x,u,u)\right), \label{3214l2e1}\\
D_+(x,u,v)&= x^2+ux(D_-(x,u,u)+D_+(x,u,u))\notag\\
&\quad+\frac{ux}{(1-vx)(u-v)}(uD_+(x,u,u)-vD_+(x,u,v))-uxD_+(x,u,u)\notag\\
&\quad+\frac{x^2}{(1-vx)(u-v)}\biggl(\frac{u^2}{(1-x)(1-u+ux)}
\left(D_+\left(x,u,\frac{1}{1-x}\right)-u(1-x)D_+(x,u,u)\right)\notag\\
&\quad+\frac{uv}{(1-x)(1-v+vx)}\left(v(1-x)D_+(x,u,v)-D_+\left(x,u,\frac{1}{1-x}\right)\right)\biggr)\notag\\
&\quad-\frac{ux^2}{(1-x)(1-u+ux)}\left(D_+\left(x,u,\frac{1}{1-x}\right)-u(1-x)D_+(x,u,u)\right). \label{3214l2e2}
\end{align}
\end{lemma}

Upon substituting \eqref{3214l2e1} into \eqref{3214l2e2}, and simplifying, we obtain the following functional equation for $D_+(x,u,v)$:
\begin{align}
&\frac{uv^2x^2+v^3x(1-x)+uv(1-x)-v^2-u+v}{(u-v)(1-vx)(1-v+vx)}D_+(x,u,v)\nonumber\\
&=\frac{u^2x(u-1)}{(u-v)(1-vx)(1-u+ux)}D_+(x,u,u)\nonumber\\
&\quad-\frac{ux^2}{(1-x)(1-vx)(1-u+ux)(1-v+vx)}D_+\left(x,u,\frac{1}{1-x}\right)-x^2.\label{eq3214a}
\end{align}

In order to solve \eqref{eq3214a}, let $v_i=v_i(x,u)$ for $0 \leq i \leq 2$ denote the roots of  $x(1-x)v^3-(1-ux^2)v^2+(1+u-ux)v-u=0$, which have series expansions
\begin{align*}
v_0&=x^{-1}-u-u(u+1)x-u(2u^2+u+2)x^2+\cdots,\\
v_1&=u+u^2x+\frac{u^3(2u-1)}{u-1}x^2+\frac{u^4(5u-2)}{u-1}x^3+\cdots,\\
v_2&=1+x-\frac{1}{u-1}x^2-\frac{2u+1}{u-1}x^3+\cdots,
\end{align*}
assuming $0<|u|,|x|<1$.  By letting $v=v_1$ or $v=v_2$ in \eqref{eq3214a}, and solving the resulting system for  $D_+(x,u,u)$ and $D_+\left(x,u,\frac{1}{1-x}\right)$, we obtain
\begin{align*}
D_+(x,u,u)&=\frac{(u-v_1)(u-v_2)x(1-x(1-x)(v_1+v_2))}{u^2(1-u)},\\
D_+\left(x,u,\frac{1}{1-x}\right)&=
-\frac{(1-x)(1-v_1+v_1x)(1-v_2+v_2x)(1+x(u-v_1-v_2))}{u}.
\end{align*}
Substituting these expressions into \eqref{eq3214a} yields
\begin{align}
D_+(x,u,v)
&=\frac{(v-v_1)(v-v_2)((u-v-v_1-v_2)x(x-1)-1)x^2}
{v^2(u-v)x^2+v(v^2-u)x-v^2+uv-u+v}.\label{3214b}
\end{align}

Let $D(x,u,v)=\sum_{n\geq 2}\sum_{i=1}^n\sum_{j=1,j\neq i}^nd_{n,i,j}u^{i-1}v^{j-i-1}x^n$.
\begin{theorem}\label{3214thm}
The generating function $D(x,u,v)$ is given by $D(x,u,v)=D_+(x,u,v)+D_-(x,u,v)$, where $D_+(x,u,v)$ and $D_-(x,u,v)$ are as in \eqref{3214b} and \eqref{3214l2e1}, respectively.
\end{theorem}

Using the facts $v_0v_1v_2=\frac{u}{x(1-x)}$ and $v_0+v_1+v_2=\frac{1-ux^2}{x(1-x)}$ in \eqref{3214b}, one can obtain a formula for $D(x,u,v)$ involving only $v_0$.

\begin{theorem}\label{3214thm2}
We have $$D(x,u,v)=D_+(x,u,v)+\frac{u^2x}{v(1-v+ux)}\left(\frac{1}{v-ux}D_+\left(x,u,\frac{u}{v-ux}\right)-D_+(x,u,u)\right),$$ where
\begin{align*}
D_+(x,u,v)=\frac{x^2 (v x - v_0 x + u - v + v_0)
          (u v v_0 x^2 - v^2 v_0 x^2 - v v_0^2 x^2 + v^2 v_0 x + v v_0^2 x - v v_0 + u)}{v_0(x-1)(u v^2 x^2 - v^3 x^2 + v^3 x - u v x + u v - v^2 - u + v)}.
\end{align*}
\end{theorem}

Taking $u=v=1$ in Theorem \ref{3214thm2}, and observing $v_0\mid_{u=1}=-\frac{3}{2x}+\frac{1}{2x}\sqrt{\frac{1-5x}{1-x}}$, we have
$$D_+(x,1,1)=\frac{x(1+x-4x^2+2x^3-\sqrt{1-6x+5x^2})}{2(1-x)(2-x)}$$
and
$$D_-(x,1,1)=\frac{x(1-3x+2x^2-\sqrt{1-6x+5x^2})}{2(2-x)},$$
which yields the following result.

\begin{corollary}\label{3214cor}
The generating function $\sum_{n\geq 2}p_n(3214)x^n$ is given by $$D(x,1,1)=\frac{x(1-x-\sqrt{1-6x+5x^2})}{2(1-x)},$$
and hence $p_n(3214)=\sum_{k=0}^{n-2}\binom{n-2}{k}C_k$ for all $n \geq 2$.
\end{corollary}

We now consider the cases of avoiding 4213 or 4312, which may be treated together. Let $e_n=p_n(4213)$ for $n \geq 1$.  To determine $e_n$, we consider a refinement of $e_n$ as follows.  Given $n \geq 3$ and $3 \leq i \leq n$, let $\mathcal{E}_{n,i}$ denote the subset of $\mathcal{E}_n=\mathcal{P}_n(4213)$ whose members have first descent top $i$ and let $\mathcal{E}_{n,i}^*\subseteq\mathcal{E}_{n,i}$ consist of those members that start $1i2$. Let $e_{n,i}=|\mathcal{E}_{n,i}|$ and $e_{n,i}^*=|\mathcal{E}_{n,i}^*|$.  For example, when $n=5$ and $i=4$, then we have $e_{5,4}=4$, the enumerated members of $\mathcal{E}_5$ comprising $\{12435,13425,14235,14253\}$, and $e_{5,4}^*=2$, which accounts for $\{14235,14253\}$.  Note $e_{n,i}$ and $e_{n,i}^*$  may assume nonzero values only for $n \geq 3$ and $i \in [3,n]$.

The $e_{n,i}$ and $e_{n,i}^*$ satisfy the following system of recurrences.

\begin{lemma}\label{4213lem1}
We have
\begin{equation}\label{4213lem1e1}
e_n=1+\sum_{i=3}^ne_{n,i}, \qquad n \geq 3,
\end{equation}
\begin{equation}\label{4213lem1e2}
e_{n,i}=\sum_{j=0}^{i-3}\binom{i-2}{j}e_{n-j,i-j}^*, \qquad 3 \leq i \leq n,
\end{equation}
and
\begin{equation}\label{4213lem1e3}
e_{n,i}^*=\sum_{m=1}^{i-3}e_{n-m,i-m}^*+\sum_{k=i+1}^n\sum_{j=0}^{k-i-1}\sum_{m=2}^{i-2}\binom{k-i-1}{j}e_{n-m-j,k-m-j}^*, \qquad 5 \leq i \leq n-1,
\end{equation}
with $e_1=e_2=1$,  $e_{n,3}^*=e_{n-2}$ for $n \geq 3$, $e_{n,4}^*=e_{n-2}$ and $e_{n,n}^*=2^{n-4}$ for $n \geq 4$.  Moreover, the same recurrences and initial conditions are seen to hold if the arrays $e_{n,i}$ and $e_{n,i}^*$ are defined in conjunction with the pattern 4312 instead of 4213, and hence, in particular, $p_n(4312)=p_n(4213)$ for all $n \geq 1$
\end{lemma}
\begin{proof}
Formula \eqref{4213lem1e1} follows from the definitions, where the extra added one prior to the sum accounts for the no descents permutation $12\cdots n$.  A member $\pi \in \mathcal{E}_{n,i}$ for $3 \leq i \leq n$ is expressible as $\pi=1\alpha i2\beta$, where $\alpha$ consists of elements of a subset of $[2,i-1]$ in increasing order.  Note that this subset must be proper or else $i$ would not be a descent top.  The letters in $\alpha$ are extraneous concerning the avoidance of 4213 as they are followed by $i$ and thus may be deleted, leaving an arbitrary member of $\mathcal{E}_{n-j,i-j}^*$ after standardization, where $j=|\alpha|$.  Considering all possible $0 \leq j \leq i-3$ then implies \eqref{4213lem1e2}.

Deletion of the first two letters of a member of $\mathcal{E}_{n,3}^*$ or $\mathcal{E}_{n,4}^*$ results in a member of $\mathcal{E}_{n-2}$ with no restrictions, whence $e_{n,3}^*=e_{n,4}^*=e_{n-2}$.  Further, $\pi \in \mathcal{E}_{n,n}^*$ implies $\pi=1n\pi'$, where $\pi'$ is $213$-avoiding.  In \cite[Theorem 4]{NR}, it was shown $p_n(213)=2^{n-2}$ for all $n \geq 2$, which implies $e_{n,n}^*=2^{n-4}$ for $n \geq 4$.  Finally, suppose $\pi \in \mathcal{E}_{n,i}^*$, where $n \geq 6$ and $5 \leq i \leq n-1$, and suppose $\pi=1i2x\pi'$, where $x \geq 3$.  We consider several cases on $x$.  If $x=3$, then 3 may clearly be deleted, giving $e_{n-1,i-1}^*$ possibilities.  Now assume $x\in [4,i-1]$. If $x<i-1$, then $x$ must be directly followed by $x+1,x+2,\ldots,i-1$, for otherwise the subsequence $ix3(x+1)$ is an occurrence of 4213.  In particular, if $x \in [4,i-1]$, then $i-1$ must occur somewhere between 2 and 3 within $\pi$.  If 3 directly follows $i-1$, then the letters $x,x+1,\ldots,i-2$ as well as 1 and $i$ may all be deleted, and thus we get $\sum_{m=2}^{i-3}e_{n-m,i-m}^*$ possibilities for $\pi$, upon allowing $x$ to vary.  Combining this with the prior case wherein $x=3$ accounts for the first summation on the right-hand side of \eqref{4213lem1e3}.

So assume either $x\in [i+1,n]$ or $x \in [4,i-1]$, with at least one letter in $[i+1,n]$ occurring prior to 3 in the latter case.  Let $k \in [i+1,n]$ denote the predecessor of 3 in either case.  Then we may write $\pi=1i2\alpha\beta k3\pi'$, where $\alpha$, $\beta$ and $\pi'$ are possibly empty and $\alpha$ and $\beta$ are increasing and consist of letters in $[4,i-1]$ and $[i+1,k-1]$, respectively.  Let $m=|\alpha|$ and $j=|\beta|$, where $0 \leq m \leq i-4$ and $0 \leq j \leq k-i-1$.  Then we must have $\alpha=(i-m)(i-m+1)\cdots(i-1)$ in order to avoid 4213, with $\binom{k-i-1}{j}$ possibilities for $\beta$.  Note that all of the letters in $\alpha$ and $\beta$ as well as 1 and $i$ may be deleted, as they place no further restriction on $\pi'$ with regard to 4213 due to $k$ directly following them.  Thus, the subsequence $2k3 \pi'$ is seen to correspond to an arbitrary member of $\mathcal{E}_{n-m-j-2,k-m-j-2}^*$.  Conversely, any $\pi$ of the stated form satisfying all of the requirements given above concerning $\alpha$, $\beta$ and $\pi'$ is seen to avoid 4213.  Considering all possible $k$, $j$ and $m$, and replacing $m$ by $m-2$, accounts for the second summation in \eqref{4213lem1e3} and thus completes the proof of \eqref{4213lem1e3}.

Concerning the final statement, we leave it as an exercise for the reader to verify that the same recurrences and initial conditions apply to the arrays $e_{n,i}$ and $e_{n,i}^*$ if they are defined in conjunction with the pattern 4312 instead of 4213.  The argument follows a similar course.  If desired, it is also possible to define a recursive bijection between $\mathcal{E}_n$ and $\mathcal{P}_n(4312)$ drawing upon the various cases used above in proving \eqref{4213lem1e3}.
\end{proof}

Let $E(x,v)=\sum_{n\geq3}\sum_{i=3}^ne_{n,i}v^{i-3}x^n$, $F(x,v)=\sum_{n\geq3}\sum_{i=3}^ne_{n,i}^*v^{i-3}x^n$ and $G(x)=\sum_{n\geq2}e_nx^n$. We obtain the following system of functional equations.

\begin{lemma}\label{4213lem2}
We have
\begin{align}
G(x)&=\frac{x^2}{1-x}+E(x,1), \label{4213lem2e1}\\
E(x,v)&=\frac{1}{(1-vx)^2}F\left(x,\frac{v}{1-vx}\right), \label{4213lem2e2}\\
\left(1-\frac{vx}{1-vx}+\frac{v^2x^2}{(1-vx)(1-v+vx)}\right)F(x,v)&=x^3(1-vx)+x^2(1+v-vx)G(x)\notag\\
&\quad+\frac{vx^2}{1-x}\left(\frac{1}{(1-vx)(1-v+vx)}-1\right)F\left(x,\frac{1}{1-x}\right).\label{4213lem2e3}
\end{align}
\end{lemma}
\begin{proof}
Equation \eqref{4213lem2e1} follows from \eqref{4213lem1e1} and the definitions of the respective generating functions.  From \eqref{4213lem1e2}, we get
\begin{align*}
E(x,v)&=\sum_{n\geq3}\sum_{i=3}^n\sum_{j=0}^{i-3}\binom{i-2}{j}e_{n-j,i-j}^*v^{i-3}x^n=\sum_{j\geq0}\sum_{n \geq j+3}\sum_{i=j+3}^n\binom{i-2}{j}e_{n-j,i-j}^*v^{i-3}x^n\\
&=\sum_{j\geq0}\sum_{n\geq3}\sum_{i=3}^n\binom{i+j-2}{j}e_{n,i}^*v^{i+j-3}x^{n+j}\\
&=\sum_{n\geq3}\sum_{i=3}^ne_{n,i}^*v^{i-3}x^n\cdot\frac{1}{(1-vx)^{i-1}}=\frac{1}{(1-vx)^2}F\left(x,\frac{v}{1-vx}\right),
\end{align*}
which yields \eqref{4213lem2e2}.  For \eqref{4213lem2e3}, first note that by the formulas for $e_{n,i}^*$ when $i=2,\,3\,$ or $n$, we have
\begin{align}
\sum_{n\geq6}\sum_{i=5}^{n-1}e_{n,i}^*v^{i-3}x^n&=F(x,v)-\sum_{n\geq3}e_{n-2}x^n-v\sum_{n\geq4}e_{n-2}x^n-\sum_{n\geq5}2^{n-4}v^{n-3}x^n\notag\\
&=F(x,v)-x^3-x^2(1+v)G(x)-\frac{2v^2x^5}{1-2vx}. \label{4213lem2e4}
\end{align}

On the other hand, multiplying both sides of \eqref{4213lem1e3} by $v^{i-3}x^n$, and summing over all $5 \leq i \leq n-1$, we get
\begin{align}
\sum_{n\geq6}\sum_{i=5}^{n-1}e_{n,i}^*v^{i-3}x^n&=\sum_{n\geq6}\sum_{i=5}^{n-1}\sum_{m=1}^{i-3}e_{n-m,i-m}^*v^{i-3}x^n\notag\\
&\quad+\sum_{n\geq6}\sum_{i=5}^{n-1}\sum_{k=i+1}^n\sum_{j=0}^{k-i-1}\sum_{m=2}^{i-2}\binom{k-i-1}{j}e_{n-m-j,k-m-j}^*v^{i-3}x^n. \label{4213lem2e5}
\end{align}
Working separately on the last two sums, we have that the first sum is given by
\begin{align*}
\sum_{n\geq6}\sum_{i=5}^{n-1}\sum_{m=1}^{i-3}e_{n-m,i-m}^*v^{i-3}x^n&=\sum_{n \geq 6}\sum_{m=1}^{n-4}\sum_{i=m+3}^{n-1}e_{n-m,i-m}^*v^{i-3}x^n-v\sum_{n\geq6}e_{n-1,3}^*x^n\\
&=\sum_{m\geq1}\sum_{n \geq m+4}\sum_{i=m+3}^{n-1}e_{n-m,i-m}^*v^{i-3}x^n-v\sum_{n\geq5}e_{n-1,3}^*x^n\\
&=\sum_{m\geq1}\sum_{n\geq4}\sum_{i=3}^{n-1}e_{n,i}^*v^{i+m-3}x^{n+m}-v\sum_{n\geq5}e_{n-3}x^n
\end{align*}
\begin{align*}
\qquad\qquad\qquad\qquad&=\sum_{n\geq4}\sum_{i=3}^{n-1}e_{n,i}^*v^{i-3}x^n\cdot \frac{vx}{1-vx}-vx^3G(x)\\
&=\frac{vx}{1-vx}\left(F(x,v)-x^3-\frac{vx^4}{1-2vx}\right)-vx^3G(x).
\end{align*}

For the second sum on the right side of \eqref{4213lem2e5}, we get
\begin{align*}
&\sum_{n\geq6}\sum_{i=5}^{n-1}\sum_{k=i+1}^n\sum_{j=0}^{k-i-1}\sum_{m=2}^{i-2}\binom{k-i-1}{j}e_{n-m-j,k-m-j}^*v^{i-3}x^n\\
&=\sum_{n \geq 6}\sum_{j=0}^{n-6}\sum_{i=5}^{n-j-1}\sum_{m=2}^{i-2}\sum_{k=i+j+1}^n\binom{k-i-1}{j}e_{n-m-j,k-m-j}^*v^{i-3}x^n\\
&=\sum_{j\geq0}\sum_{n\geq j+6}\sum_{m=2}^{n-j-3}\sum_{i=m+2}^{n-j-1}\sum_{k=i+j+1}^n\binom{k-i-1}{j}e_{n-m-j,k-m-j}^*v^{i-3}x^n-S,
\end{align*}
where $S=v\sum_{j\geq0}\sum_{n\geq j+6}\sum_{k=j+5}^n\binom{k-5}{j}e_{n-j-2,k-j-2}x^n$, upon adding and subtracting to the summation the case when $i=4$.  Considering separately the cases when $m\geq 3$ or $m=2$, one obtains
\begin{align*}
&\sum_{j\geq0}\sum_{m\geq3}\sum_{n\geq j+m+3}\sum_{i=m+2}^{n-j-1}\sum_{k=i+j+1}^n\binom{k-i-1}{j}e_{n-m-j,k-m-j}^*v^{i-3}x^n\\
&\quad+\sum_{j\geq0}\sum_{n\geq j+6}\sum_{i=4}^{n-j-1}\sum_{k=i+j+1}^n\binom{k-i-1}{j}e_{n-j-2,k-j-2}^*v^{i-3}x^n-S\\
&=\sum_{j\geq0}\sum_{m\geq3}\sum_{i \geq m+2}\sum_{n\geq i+j+1}\sum_{k=i+j+1}^n\binom{k-i-1}{j}e_{n-m-j,k-m-j}^*v^{i-3}x^n\\
&\quad+\sum_{j\geq0}\sum_{i \geq 5}\sum_{n \geq i+j+1}\sum_{k=i+j+1}^n\binom{k-i-1}{j}e_{n-j-2,k-j-2}^*v^{i-3}x^n\\
&=\sum_{j\geq0}\sum_{m\geq3}\sum_{i \geq2}\sum_{n \geq i+1}\sum_{k=i+1}^n\binom{k+j-i-1}{j}e_{n,k}^*v^{i+m-3}x^{n+m+j}\\
&\quad+\sum_{j\geq0}\sum_{i \geq 5}\sum_{n \geq i-1}\sum_{k=i-1}^n\binom{k+j-i+1}{j}e_{n,k}^*v^{i-3}x^{n+j+2}\\
&=\frac{x^3}{1-vx}\sum_{i\geq2}\sum_{n\geq i+1}\sum_{k=i+1}^ne_{n,k}^*v^ix^n\cdot\frac{1}{(1-x)^{k-i}}+\sum_{i\geq 5}\sum_{n \geq i-1}\sum_{k=i-1}^ne_{n,k}^*v^{i-3}x^{n+2}\cdot\frac{1}{(1-x)^{k-i+2}}.
\end{align*}

Interchanging the sum in $i$ with the others now gives
\begin{align*}
&\frac{x^3}{1-vx}\sum_{n\geq 3}\sum_{k=3}^ne_{n,k}^*\frac{x^n}{(1-x)^k}\sum_{i=2}^{k-1}(v(1-x))^i+\sum_{n\geq4}\sum_{k=4}^ne_{n,k}^*\frac{x^{n+2}}{v(1-x)^k}\sum_{i=5}^{k+1}(v(1-x))^{i-2}\\
&=\frac{x^3}{1-vx}\sum_{n\geq 3}\sum_{k=3}^ne_{n,k}^*\frac{x^n}{(1-x)^k}\cdot\frac{(v(1-x))^2-(v(1-x))^k}{1-v+vx}\\
&\quad+\sum_{n\geq3}\sum_{k=3}^ne_{n,k}^*\frac{x^{n+2}}{v(1-x)^k}\cdot\frac{(v(1-x))^3-(v(1-x))^k}{1-v+vx}\\
&=\frac{v^2x^3}{(1-vx)(1-v+vx)}\left(\frac{1}{1-x}F\left(x,\frac{1}{1-x}\right)-vF(x,v)\right)+\frac{v^2x^2}{1-v+vx}\left(F\left(x,\frac{1}{1-x}\right)-F(x,v)\right).\\
\end{align*}
Comparing \eqref{4213lem2e4} and \eqref{4213lem2e5}, and combining the prior calculations, we have
\begin{align*}
&F(x,v)-x^3-x^2(1+v)G(x)-\frac{2v^2x^5}{1-2vx}\\
&=\frac{vx}{1-vx}\left(F(x,v)-x^3-\frac{vx^4}{1-2vx}\right)-vx^3G(x)\\
&\quad+\frac{v^2x^3}{(1-vx)(1-v+vx)}\left(\frac{1}{1-x}F\left(x,\frac{1}{1-x}\right)-vF(x,v)\right)+\frac{v^2x^2}{1-v+vx}\left(F\left(x,\frac{1}{1-x}\right)-F(x,v)\right).
\end{align*}
The last equality may be rewritten as \eqref{4213lem2e3}, which completes the proof.
\end{proof}

\begin{theorem}\label{4213th}
We have
\begin{align}
E(x,v)&=\frac{x^3(1-v-3vx+2v^2x+v^2x^2)+x^2(1-3vx+v^2x+v^2x^2)G(x)}{(1-vx)^2(1-v-3vx+3v^2x+v^2x^2)},\label{4213the0}
\end{align}
where $G(x)=\frac{x(1-x-\sqrt{1-6x+5x^2})}{2(1-x)}$.  In particular, the formula for the generating function $\sum_{n\geq2}p_n(\tau)x^n$, where $\tau \in \{4213,4312\}$, is given by $G(x)=E(x,1)+\frac{x^2}{1-x}$.
\end{theorem}
\begin{proof}
The second statement follows from the first, upon taking $v=1$, so we need only prove the first.  By \eqref{4213lem2e2} when $v=1$ and \eqref{4213lem2e1}, we have
\begin{equation}\label{4213the1}
F\left(x,\frac{1}{1-x}\right)=(1-x)^2G(x)-x^2(1-x).
\end{equation}
Substituting \eqref{4213the1} into \eqref{4213lem2e3}, and simplifying, we get
\begin{align}
\left(1-\frac{vx}{1-vx}+\frac{v^2x^2}{(1-vx)(1-v+vx)}\right)F(x,v)&=\frac{x^3(1-vx)(1-v+vx)-vx^4}{(1-vx)(1-v+vx)}\notag\\
&\quad+\frac{x^2((1-vx)(1-v+vx)+v-vx)}{(1-vx)(1-v+vx)}G(x). \label{4213the2}
\end{align}
Taking $v=v_0=\frac{1+x-\sqrt{1-6x+5x^2}}{2x(2-x)}$ in \eqref{4213the2} is seen to cancel out the left-hand side, and solving for $G(x)$, we find
$$G(x)=-\frac{x(1-v_0x)(1-v_0+v_0x)-v_0x^2}{(1-v_0x)(1-v_0+v_0x)+v_0(1-x)}.$$
Using the relation $(1-v_0x)(1-v_0+v_0x)=v_0x(1-v_0)$ gives
$$G(x)=-\frac{v_0x^2(1-v_0)-v_0x^2}{v_0x(1-v_0)+v_0(1-x)}=\frac{v_0x^2}{1-v_0x}=\frac{x(1-x-\sqrt{1-6x+5x^2})}{2(1-x)}.$$
Note that the choice of the minus sign prior to the radical term in $v_0$ is dictated by the requirement that $G(x)/x$ evaluated at $x=0$ be zero. Substituting this expression for $G(x)$ back into \eqref{4213the2}, solving for $F(x,v)$ and using \eqref{4213lem2e2} yields \eqref{4213the0}, which completes the proof.
\end{proof}

\section{Further results}

In this section, we discuss further results concerning the avoidance of a single pattern of length four by members of $\mathcal{P}_n$.  We start with the following general result, wherein $\tau'=1(\tau+1)$ and $\tau+1$ is obtained from $\tau$ by adding one to each entry of $\tau$.

\begin{proposition}\label{genprop}
Let $\tau$ be a pattern.  Then $p_n(\tau)=p_n(\tau')$ for all $n \geq 1$ if and only if $\tau$ does not start with $1$.
\end{proposition}
\begin{proof}
For the sufficiency, suppose $\pi \in \mathcal{P}_n$ contains an occurrence of $\tau$ witnessed by the subsequence $xyz$.  Then $\tau$ not starting with 1 implies $x \neq 1$ and hence $1xyz$ is an occurrence of $\tau'$.  This implies indeed $\mathcal{P}_n(\tau)=\mathcal{P}_n(\tau')$, and hence $p_n(\tau)=p_n(\tau')$, for all $n \geq 1$.

For the necessity, we show that if $\tau$ starts with 1, then there exists $n \geq 1$ such that $p_n(\tau)\neq p_n(\tau')$.  Clearly, we may assume $\tau$ has length at least two.  Let $m$ be the smallest $n$ such that $p_n(\tau)\neq B_{n-1}$, i.e., there exists a member of $\mathcal{P}_n$ that contains an occurrence of $\tau$.  Note that $m$ exists for any pattern $\tau$.  We wish to show $p_m(\tau')=B_{m-1}$, which would establish the necessity.  Suppose, to the contrary, that $p_m(\tau')\neq B_{m-1}$ and let $\pi \in \mathcal{P}_m$ contain an occurrence $s$ of $\tau'$.  We may assume that the role of `1' in $s$ is played by the actual letter $1$ and let $x$ denote the letter of $\pi$ playing the role of the `2' in $s$.  First suppose $x$ corresponds to a descent bottom of $\pi$.  Then deletion of the first run of $\pi$ followed by standardization yields a member of $\mathcal{P}_k$ that contains an occurrence of $\tau$ for some $k<m$, since the letters of $\pi$ excluding its first run would contain $\tau$ in this case.  But then $p_k(\tau)\neq B_{k-1}$, which contradicts the minimality of $m$.  So assume $x$ does not correspond to a descent bottom of $\pi$.  Then it is seen that one may delete $x$ in this case and remain a flattened partition as $x$ does not start a run.  The resulting member of $\mathcal{P}_{m-1}$ contains $\tau$, as witnessed by the remaining letters of $s$ (the first of which is assumed to be 1).  Again, this contradicts the minimality of $m$, and hence we must have $p_m(\tau')=B_{m-1}$, as desired.
\end{proof}

\begin{table}[h!]
	\centering
	\renewcommand{\arraystretch}{1.3}
	\begin{tabular}{|c|c|c|c|}
		\hline
		\textbf{Class} & \textbf{Pattern $\tau$} & \textbf{$p_n(\tau)$} & \textbf{OEIS sequence} \\
		\hline
		$1$ & $1234$ &$1, 1, 2, 4, 5, 3, 0, 0, 0, 0,  \ldots.$ & Trivial \\
		\hline
		$2$ & $1243$ & $1, 1, 2, 4, 7, 11, 16, 22, 29, 37, \ldots.$ & \texttt{A152947} \\
		\hline
		$3$ & \begin{tabular}{c}$1324$\\$1423$\end{tabular} & $1, 1, 2, 4, 8, 16, 32, 64, 128, 256, \ldots.$ & \texttt{A011782} \\
		\hline
		$4$ & $1342$ &	$1, 1, 2, 4, 9, 21, 51, 127, 323, 835, \ldots.$ & \texttt{A001006} \\
		\hline
		$5$ & $3124$ &	$1, 1, 2, 5, 14, 40, 113, 315, 871, 2400, \ldots $ & \texttt{Unknown} \\
		\hline
        $6$ & $4123$ &	$1, 1, 2, 5, 14, 40, 114, 324, 920, 2612, \ldots$ &		\texttt{A159035}  \\
		\hline
		$7$ & $2134$ &	$1, 1, 2, 5, 14, 40, 116, 341, 1014, 3045, \ldots$ & \texttt{A273900}\\
		\hline
		$8$ & \begin{tabular}{c}$1432$\\$2143$\\$2314$\\$2413$\\$3142$\\$3412$\\$4132$ \end{tabular}  &
		$1, 1, 2, 5, 14, 42, 132, 429, 1430, 4862, \ldots.$ & \texttt{A000108} \\
		\hline
		$9$ & $2341$ & $1, 1, 2, 5, 14, 44, 148, 528, 1972, 7647, \ldots $&	\texttt{A307786} (?) \\
		\hline
       $10$ & \begin{tabular}{c}$3214$\\$4213$\\$4312$\end{tabular} &
		$1, 1, 2, 5, 15, 51, 188, 731, 2950, 12235, \ldots$ &
		\texttt{A007317} \\
		\hline
		$11$ & \begin{tabular}{c}$2431$\\$3241\,(?)$\end{tabular} &	$1, 1, 2, 5, 15, 51, 190, 758, 3187, 13963, \ldots$ & \texttt{Unknown} \\
		\hline
       $12$ & $3421$ &	$1, 1, 2, 5, 15, 51, 190, 759, 3203, 14129, \ldots $ &
		\texttt{Unknown} \\
		\hline
		
		$13$ & $4231$ &	$1, 1, 2, 5, 15, 51, 190, 759, 3203, 14130, \ldots$ &\texttt{Unknown} \\
		\hline
      $14$ & $4321$ &	$1, 1, 2, 5, 15, 52, 202, 858, 3910, 18846, \ldots$ & \texttt{Unknown} \\
		\hline
	\end{tabular}
	\caption{Wilf-equivalence classes and avoidance sequences (for $1 \leq n \leq 10$) for patterns of length four}
	\label{tab:we-classes}
\end{table}

In Table \ref{tab:we-classes} above, we present the Wilf-equivalence classes corresponding to the avoidance of a single pattern of length four by distinct flattened partitions.  They are arranged in increasing order with respect to the growth rate of the corresponding avoidance counting sequences.  Combining results from Table \ref{tab:length3-patterns} above with Proposition \ref{genprop} accounts for the facts witnessed in Table \ref{tab:we-classes} that $p_n(1324)=p_n(1423)=2^{n-2}$ for all $n \geq 2$, $p_n(1342)=M_{n-1}$ and $p_n(1432)=C_{n-1}$ for $n \geq 1$. Further, we have $p_n(4132)=p_n(321)=C_{n-1}$, for if $\pi \in \mathcal{P}_n$ contains an occurrence of 321 witnessed by the subsequence $xyz$, then $xwyz$ is a 4132, where $w$ denote the first letter in the run containing $y$.  Combining these observation with Theorems \ref{Catalan} and \ref{trCatalan} above accounts for all of the Wilf-classes of size greater than one except for the possible equivalence observed between 2431 and 3241.  We state this last equivalence as a conjecture below, which has been confirmed up to $n=14$ numerically.  All of the other Wilf-classes are singletons whose various counting sequences are differentiated on numerical grounds requiring at most $n=10$.  Further, we have that 2341 is a singleton class wherein $p_n(2341)$ appears to correspond to the $(n-1)$-st term of entry A307786 in \cite{Sloane} for all $n\geq 1$, which has been confirmed up to $n=15$.

\begin{conjecture}\label{conj1} If $n \geq 1$, then $p_n(2431)=p_n(3241)$.
\end{conjecture}

\begin{conjecture}If $n \geq 1$, then $p_n(2341)=A307786(n-1)$.
\end{conjecture}

Note that the results contained herein combined with Conjecture \ref{conj1} would complete the Wilf-classification for the avoidance of a single pattern of length four.

One may verify directly the first six terms of the sequence $p_n(1234)$ given in Table \ref{tab:we-classes} and that $p_n(1234)=0$ for all $n \geq 7$.  Indeed, for the pattern $\tau=12\cdots r$, where $r \geq 2$, we have that $\binom{r}{2}$ is the largest $n$ such that $p_n(\tau)\neq0$.  To see this, first note that $\pi=1\alpha_{r-2}2\alpha_{r-3}3\cdots\alpha_1(r-1)$, where
$$\alpha_j=\left(r+\binom{j}{2}\right)\left(r+\binom{j}{2}+1\right)\cdots\left(r+\binom{j+1}{2}-1\right), \qquad 1 \leq j \leq r-2,$$
is seen to avoid $\tau$.  Further, it is apparent from the construction of $\pi$ that any flattened partition of length greater than $|\pi|=\binom{r}{2}$ must contain $\tau$, and hence $p_n(\tau)=0$ for all $n>\binom{r}{2}$.  By an elementary enumeration, one can show $p_n(1243)=\binom{n-1}{2}+1$ for all $n \geq 1$, and hence we see that the pattern pairs $123/1234$ and $132/1243$ demonstrate the necessity of Proposition \ref{genprop}.

Concerning the determination of $p_n(\tau)$ in other cases where $\tau$ comprises its own Wilf-class, the problems are of varying degrees of difficulty.  The patterns ending in 1 seem to be more difficult to enumerate than the others, especially 4321.  We conclude with the following result where we enumerate the cases 3124, 4123 and 2134, which completes the enumeration of all the patterns of length four not ending in 1.  We again make use of systems of functional equations in establishing these cases.

\begin{theorem}\label{othercases}
Let $r_n=p_n(3124)$, $s_n=p_n(4123)$ and $t_n=p_n(2134)$ for $n \geq 1$.  Then we have
\begin{equation}\label{3124te1}
r_n=5r_{n-1}-8r_{n-2}+5r_{n-3}, \qquad n \geq 5,
\end{equation}
\begin{equation}\label{4123te1}
s_n=4s_{n-1}-4s_{n-2}+2s_{n-3}, \qquad n \geq 5,
\end{equation}
and
\begin{equation}\label{2134te1}
\sum_{n\geq 2}t_nx^n=\frac{(1-x)(1-3x+x^2-x^3)-(1-x)^2\sqrt{1-4x+2x^2+x^4}}{2x^2},
\end{equation}
where $r_n$ and $s_n$ satisfy the same initial conditions $r_1=r_2=1$, $r_3=2$, $r_4=5$.
\end{theorem}
\begin{proof}
We demonstrate only \eqref{3124te1} and outline its proof as follows.  Given $n \geq 2$ and $i,j \in [n]$ with $i \neq j$, let $r_{n,i,j}$ denote the number of members of $\mathcal{P}_n(3124)$ ending in $i,j$.  Then one can argue combinatorially the recurrences
\begin{align}
r_{n,i,j}&=\sum_{m=0}^{i-j-1}\sum_{s=1}^{j-1}r_{n-m-1,s,i-m-1}, \quad 2 \leq j< i \leq n, \label{2134e1}\\
r_{n,i,i+1}&=\sum_{s=1,s\neq i}^{n-1}r_{n-1,s,i}, \qquad 2 \leq i \leq n-1, \label{2134e2}\\
r_{n,i,j}&=\sum_{k=j}^n\sum_{s=1}^{i-1}r_{n-k+i,s,i}, \qquad 2 \leq i<j-1\leq n-1,\label{2134e3}
\end{align}
with $r_{2,1,2}=1$, $r_{2,2,1}=0$ and $r_{n,i,1}=r_{n,1,j}=0$ for all $i$ and $j$ if $n \geq 3$. Define the generating functions $$R_+(x,u,v)=\sum_{n\geq2}\sum_{i=1}^{n-1}\sum_{j=i+1}^{n}r_{n,i,j}u^{i-1}v^{j-i-1}x^n$$ and $$R_-(x,u,v)=\sum_{n\geq2}\sum_{i=2}^n\sum_{j=1}^{i-1}r_{n,i,j}u^{i-1}v^{j-i-1}x^n.$$
Translating \eqref{2134e1}--\eqref{2134e3}, we have
\begin{align}
R_+(x,u,v)&=x^2+ux(R_+(x,u,u)+R_-(x,u,u))+\frac{uvx^2}{(1-x)(1-vx)}R_+(x,u,u),\label{eq1p1}\\
R_-(x,u,v)&=\frac{u^2x}{v(1-v)(v-ux)}(R_+(x,u,u/v)-vR_+(x,u,u)).\label{eq1p2}
\end{align}

Using \eqref{eq1p2} in \eqref{eq1p1} yields
\begin{align}
R_+(x,u,v)&=x^2+\frac{ux^2}{(1-u)(1-x)}(R_+(x,u,1)-uR_+(x,u,u))+uxR_+(x,u,u)\nonumber\\
&\quad+\frac{uvx^2}{(1-x)(1-vx)}R_+(x,u,u).\label{eqp3}
\end{align}
First letting $v=1$ and then $v=u$ in \eqref{eqp3}, one obtains a system of equations in the quantities $R_+(x,u,1)$ and $R_+(x,u,u)$. Solving this system yields
\begin{align*}
R_+(x,u,1)&=-\frac{x^2(x-1)(u^2 x^2 + u x^3 - 3u x^2 + u x - x^2 + 2x - 1)}{u^2 x^4 - u x^4 + 4u x^3 - 5u x^2 + x^3 + 2u x - 3x^2 + 3x - 1},\\
R_+(x,u,u)&=-\frac{x^2(x-1)^3(u x - 1)}
{u^2 x^4 - u x^4 + 4u x^3 - 5u x^2 + x^3 + 2u x - 3x^2 + 3x - 1}.
\end{align*}
Taking $v=u$ in \eqref{eq1p2}, and using these last two expressions, we find
$$R_-(x,u,u)=\frac{x^3(u x^3 - 3u x^2 + u x - x^2 + 2x - 1)}
{u^2 x^4 - u x^4 + 4u x^3 - 5u x^2 + x^3 + 2u x - 3x^2 + 3x - 1}.$$

By the definitions, the generating function $\sum_{n\geq1}r_nx^n$ is given by
$$x+R_+(x,1,1)+R_-(x,1,1)=\frac{x(2x-1)(x-1)^2}{5x^3-8x^2+5x-1},$$
which implies recurrence \eqref{3124te1} and recovers the initial values of $r_n$.  Comparable proofs may be given for \eqref{4123te1} and \eqref{2134te1}, upon considering the parameters and distributions used above in the respective cases 4213 and 2143.  Note that from \eqref{4123te1} and \eqref{2134te1}, one obtains new combinatorial interpretations of the OEIS sequences A159035 and A273900, respectively.
\end{proof}


\end{document}